\documentclass[12pt]{article}
\usepackage{amssymb}
\usepackage{amsmath}
\usepackage{amsthm}
\usepackage{comment}								
\usepackage{graphicx}
\usepackage{xcolor}									
\renewcommand{\v}[1]{{\bf #1}}			
\newcommand{\vd}[1]{\dot{\bf #1}}
\newcommand{\f}[2]{{\bf #1}_{#2}}
\newcommand{\ff}[2]{{#1}_{#2}}

\newcommand{\sect}[1]{Sec.~\ref{#1}}

\newcommand{\eps}{\varepsilon}
\newcommand{\cO}{\mathcal{O}}
\newcommand{\sh}{\sigma}
\newcommand{\step}{\operatorname{step}}

\definecolor{thirdIterate}{rgb}{.15,.8,.15}

\newtheorem{theorem}{Theorem}[section]
\newtheorem{proposition}[theorem]{Proposition}

\begin{document}



\title{Jitter in Dynamical Systems with Intersecting Discontinuity Surfaces}
\author{
M.R.~Jeffrey$^{\dagger}$, G.~Kafanas$^{\dagger}$, and D.J.W.~Simpson$^{\ddagger}$\\
\small
$^{\dagger}$Department of Engineering Mathematics, University of Bristol, Bristol, UK\\
\small
$^{\ddagger}$Institute of Fundamental Sciences, Massey University, Palmerston North, New Zealand
}
\maketitle

\begin{abstract}


Mathematical models involving switches --- in the form of differential equations with discontinuities ---
can accomodate real-world non-idealities through perturbations by hysteresis, time-delay, discretization, and noise.
These are used to model the processes associated with switching in
electronic control, mechanical contact, predator-prey preferences, and genetic or cellular regulation.
The effect of such perturbations on rapid switching dynamics about a single switch are somewhat benign:
in the limit that the size of the perturbation goes to zero the dynamics is given by Filippov's sliding solution.
When multiple switches are involved, however, perturbations can have complicated effects, as shown in this paper.
In the zero-perturbation limit, hysteresis, time-delay, and discretization
cause erratic variation or `jitter' for stable sliding motion,
whilst noise generates a relatively regular sliding solution similar to the canopy solution
(an extension of Filippov's solution to multiple switches).
We illustrate the results with a model of a switched power circuit, 
and showcase a variety of complex phenomena that perturbations can generate,
including chaotic dynamics, exit selection, and coexisting sliding solutions.

\end{abstract}


\section{Introduction}
\label{sec:intro}
\setcounter{equation}{0}



For systems of differential equations that are piecewise-smooth --- smooth except for discontinuities on isolated surfaces --- orbits
may cross discontinuity surfaces or slide along them.
Sliding motion models frictional sticking in mechanics \cite{OeHi96}, for example,
and is used to achieve robust control via switched feedback mechanisms in electronics \cite{TaLa12,Su06}.
The definition of sliding involves certain rules and assumptions that are not part of the standard theory of dynamical systems, and which remain incompletely understood.

Filippov \cite{Fi88} showed that sliding motion can be defined by taking a convex combination of the vector field across its discontinuities.
For a single discontinuity surface this gives a unique solution,
but along the intersection of two or more discontinuity surfaces this gives a family of solutions. 
Alternatively, sliding motion can be defined by perturbing the system in a manner that generates a classical solution,
and then taking the limit as the size of the perturbation goes to zero.
For a single discontinuity surface this gives a unique solution (equivalent to Filippov's solution) for many types of perturbation.
The purpose of this paper is to explore the nature of sliding motion defined in this way for intersecting discontinuity surfaces.
Our main result is that the sliding motion may be uniquely defined but highly sensitive to variations in the vector field, causing `jitter' in the sliding solutions.

Consider an $n$-dimensional system
\begin{equation}
\vd x = \v f(\v x) \;,
\label{eq:ns}
\end{equation}
whose right-hand side $\v f:\mathbb{R}^n \to \mathbb{R}^n$ is smooth
except on discontinuity surfaces $\Sigma_1,\ldots,\Sigma_m$.
We assume that each $\Sigma_i$ divides $\mathbb{R}^n$ into two pieces,
hence $\Sigma = \Sigma_1 \cup \cdots \cup \Sigma_m$ divides $\mathbb{R}^n$
into $2^m$ open regions ${\cal R}_1, \ldots, {\cal R}_{2^m}$,
some of which may be empty or disconnected.
We then write
\begin{equation}
\v f(\v x) = \f fi(\v x) {\rm ~~for~} \v x \in {\cal R}_i \;, \qquad i=1,...,2^m,
\label{eq:nsf}
\end{equation}
where $\f f1,\ldots \f f{2^m}$ are independent smooth vector fields.

Sliding solutions to (\ref{eq:ns}) are those that evolve on at
least one discontinuity surface $\Sigma_i$ for a sustained interval of time.
About a single discontinuity surface, say $\Sigma_1$ bounding ${\cal R}_1$ and ${\cal R}_2$, the convex combination
\begin{equation}
\v f_{\rm comb}(\v x,\lambda) = \lambda \f f1(\v x) + (1-\lambda) \f f2(\v x) \;, \qquad \lambda \in [0,1] \;,
\label{eq:convexCombination}
\end{equation}
describes the convex hull of values that $\vd x$ can explore by evolving in
arbitrarily small increments of $\vd x = \f f1(\v x)$ and $\vd x = \f f2(\v x)$. 
The switching multiplier $\lambda$ can be interpreted as the proportion of time spent following $\f f1$,
and $1-\lambda$ as the proportion of time spent following $\f f2$
(this relates to the {\em duty cycle} in electronics).
Filippov's solution is found by solving for the value of $\lambda$ that gives $\v f_{\rm comb}$ tangent to $\Sigma_1$.
Letting $\sh(\v x) = 0$ represent $\Sigma_1$,
where $\sh : \mathbb{R}^n \to \mathbb{R}$ is a continuous function,
the tangency condition is $\v f(\v x,\lambda) \cdot \nabla \sh(\v x) = 0$.
Filippov's solution therefore satisfies $\vd x = \v f(\v x,\lambda_{\rm slide})$ where
$\lambda_{\rm slide} = \frac{\f f2 \cdot \nabla \sh}{(\f f2 - \f f1) \cdot \nabla \sh}$.
Recent generalizations have extended Filippov's method to permit non-convex combinations \cite{Je14d}.


Such solution methods assume an ideal model of switching.
We may choose to add more complexity to the formulation of the discontinuity for two main reasons.
First, in numerical simulations or from a theoretical viewpoint
we may seek to `regularise' (\ref{eq:ns}) on $\Sigma$
to resolve ambiguities in evolution resulting from the discontinuity.
Second, we may use alternate switching formulations in order to model complex physical processes,
such as an electronic switch \cite{UtGu99},
a change in the direction of Coulomb friction \cite{WoSt08,OlAs98},
a discontinuous change of the surface reflectivity of the Earth \cite{KaEn13},
predator switching between prey sources \cite{KuRi03,PiPo14},
phase transition of a superconductor at its critical temperature \cite{mach,JeCh10}, and so on.
A major challenge for the theory of piecewise-smooth systems is to understand the robustness
of (\ref{eq:ns}) to different kinds of regularisation.


There are a handful of regularisations that are obvious to consider.
Here we describe these for a single discontinuity surface $\Sigma_1$ between regions ${\cal R}_1$ and ${\cal R}_2$; the extension to multiple switches follows naturally.
We suppose $\Sigma_1$ is defined by $\sh(\v x) = 0$,
with $\vd x=\f f1 (\v x)$ for $\sh(\v x)>0$, and $\vd x=\f f2 (\v x)$ for $\sh(\v x)<0$. The regularisations we consider are:
\begin{enumerate}
\renewcommand{\labelenumi}{\roman	{enumi})}
\item
{\em smoothing}:
the switch between $\f f1$ and $\f f2$ is modulated by a continuous function $\psi : \mathbb{R} \to \mathbb{R}$
with $\lim_{\eps \to 0} \psi(\sh/\eps) = {\rm sgn}(\sh)$ for $\sh \neq 0$; specifically we consider
$\vd x = \psi(\sh(\v x)/\eps)\,\f f1(\v x) +
\left(1 - \psi(\sh(\v x)/\eps)\right) \,\f f2$;
\item
{\em hysteresis}:
a switch from $\f f1$ to $\f f2$ as $\sh$ decreases occurs when $\sh(\v x) = -\eps$,
while a switch from $\f f2$ to $\f f1$ as $\sh$ increases occurs when $\sh(\v x) = \eps$;
\item
{\em time-delay}:
a switch occurs at a time $\eps$ after crossing $\Sigma_1$;
\item
{\em discretization}:
the system is discretized, say via the forward Euler numerical scheme: $\v x_{k+1} = \v x_k + \eps \v f(\v x_k)$.
\item
{\em noise}:
the ordinary differential equation $\vd x=\v f(\v x)$ is converted to a stochastic differential equation
via the addition of order $\eps$ noise,
or the switching function $\sh$ is perturbed to $\sh(\v x) + \eps \xi(t)$, where $\xi(t)$ is a stochastic process.
\end{enumerate}


The effect of smoothing a discontinuity is largely understood and is in the main rather simple.
Typically (if the system is non-degenerate and absent of certain singularities)
the smooth system $\vd x = \v f(\v x;\eps)$ is equivalent
to the discontinuous system up to order $\eps$ (or order $\eps^p$ where $p>0$ for some singularities) \cite{SoTe96,TeDa12,NoJe15}. 


For an attracting section of a single discontinuity surface,
the effects of the remaining types of perturbation are also straight-forward.
With hysteresis, time-delay, or discretization by forward Euler,
one obtains `chattering' solutions that tend to Filippov's solution in the
zero-perturbation limit ($\eps \to 0$) \cite{Fi88,UtGu99}.
Similarly additive white noise perturbations yield stochastic solutions
that tend to Filippov's solution as $\eps \to 0$ \cite{BuOu09,SiKu14b}.
Near singularities the situation can be more involved --- discretizing
near grazing points in the presence of a limit cycle leads to micro-chaos \cite{GlKo10b},
and regularisations of two-fold singularities leads to a probabilistic notion of forward evolution \cite{Si14c,Je16}.
In general, away from bifurcation points and determinacy-breaking singularities, 
the dynamics about a single discontinuity surface is robust to perturbations.



The remainder of this paper is organized as follows.
We begin in \sect{sub:intDiscSurf} by introducing a general three-dimensional system with two discontinuity surfaces.
We work with this system throughout the paper as it involves 
the fewest number of dimensions and discontinuity surfaces
needed to exhibit evolution along intersecting discontinuity surfaces.
In \sect{sub:hullCanopy} we define the convex hull and the canopy solution.
In \sect{sub:zeroPertLimit} we describe the general procedure by which
sliding motion is defined through a zero-perturbation limit for hysteresis, time-delay, discretization, and noise.
In \sect{sub:simulation} we present the results of numerical simulations for each of these types of perturbation,
illustrating the wild variations typical in the sliding vector field
and that are responsible for jittery sliding solutions.

In \sect{sec:hysteresis} we study hysteresis --- the best understood perturbation after smoothing.
We first review the results of Alexander and Seidman \cite{AlSe99}
on the existence and uniqueness of the zero-hysteresis sliding solution in the most basic scenario
when all four components of the vector field are directed inwards to the intersection of the discontinuity surfaces, \sect{sub:chatterbox}.
This solution is determined by the hysteretic dynamics of a two-dimensional piecewise-constant vector field within a rectangle $\Omega$.
We show how the dynamcis within $\Omega$ is captured by a continuous, invertible, degree-one circle map, \sect{sub:circleMap}.
We then study mode-locking regions and the dependence of the zero-hysteresis sliding solution
on the relative sizes of the hysteresis bands about the two discontinuity surfaces, \sect{sub:zeroHystSoln}.
Finally we look at three pedagogic scenarios for which at least one vector field component
is not directed inwards, Secs.~\ref{sub:chaos}--\ref{sub:exitSelection}. 

In \sect{sec:other} we study the remaining three types of pertubation.
Our aim is to briefly illustrate the typical nature of sliding solutions defined by using these perturbations.
We highlight a range of complications that do not arise in the case of hysteresis.
A complete analysis and understanding is beyond the scope of this paper.

In \sect{sec:eg} we illustrate jitter in an electronic circuit.
The circuit is well-modelled as a Filippov system with two intersecting discontinuity surfaces,
but, more realistically, rapid switching occurs within two bands of hysteresis.
We show how the dynamics are well explained via the zero-hysteresis sliding solution.

Lastly we provide closing remarks in \sect{sec:conc}.

\section{The general problem and the concept of jitter}
\label{sec:general}
\setcounter{equation}{0}



\subsection{Intersecting discontinuity surfaces}
\label{sub:intDiscSurf}

Consider a system $\vd x = \v f (\v x)$ with state vector
$\v x = (x,y,z) \in \mathbb{R}^3$
and discontinuity surfaces $x = 0$ and $y = 0$, as in Fig.~\ref{fig:schemIntSwMan}.
We write
\begin{equation}
\vd x = \begin{cases}
\f f1(\v x) \;, & (x,y) \in {\cal Q}_1 \\[-1.mm]
\f f2(\v x) \;, & (x,y) \in {\cal Q}_2 \\[-1.mm]
\f f3(\v x) \;, & (x,y) \in {\cal Q}_3 \\[-1.mm]
\f f4(\v x) \;, & (x,y) \in {\cal Q}_4
\end{cases} \;,
\label{eq:f}
\end{equation}
where each $\f fi$ is smooth and
${\cal Q}_i$ denotes the $i^{\rm th}$ quadrant of the $(x,y)$-plane:
\begin{equation}
\begin{split}
{\cal Q}_1 &= \{ (x,y) ~|~ x>0, y>0 \} \;, \\
{\cal Q}_2 &= \{ (x,y) ~|~ x<0, y>0 \} \;, \\
{\cal Q}_3 &= \{ (x,y) ~|~ x<0, y<0 \} \;, \\
{\cal Q}_4 &= \{ (x,y) ~|~ x>0, y<0 \} \;.
\end{split}
\nonumber
\end{equation}
While an orbit $\v x(t)$ is following $\f fi$, we say that $\v x(t)$ is in {\em mode} $i$.
This particular definition extends nicely to
the perturbations of (\ref{eq:f}) considered below.
We also write $\f fi = (\ff fi,\ff gi,\ff hi)$ to indicate the components of $\f fi$.

The two discontinuity surfaces intersect along the $z$-axis which we denote by $\Gamma$.
At any $(0,0,z) \in \Gamma$, we say that $\f fi$ is {\em directed inwards}
if $(-\ff fi(0,0,z),-\ff gi(0,0,z)) \in {\cal Q}_i$. 
Any section of $\Gamma$ along which each of the $\f fi$ are directed inwards is {\em attracting}
in the sense that all nearby orbits, both regular and sliding, approach the $z$-axis in forward time.
Note that sections of $\Gamma$ can be attracting if one, or even all, of the $\f fi$ are not directed inwards.

If $\Gamma$ is attracting, we expect motion to ensue on $\Gamma$.
To this end we search for a way to define such {\em sliding} motion. 

\begin{figure}[b!]
\begin{center}
\setlength{\unitlength}{1cm}
\begin{picture}(10,5)
\put(0,0){\includegraphics[height=5cm]{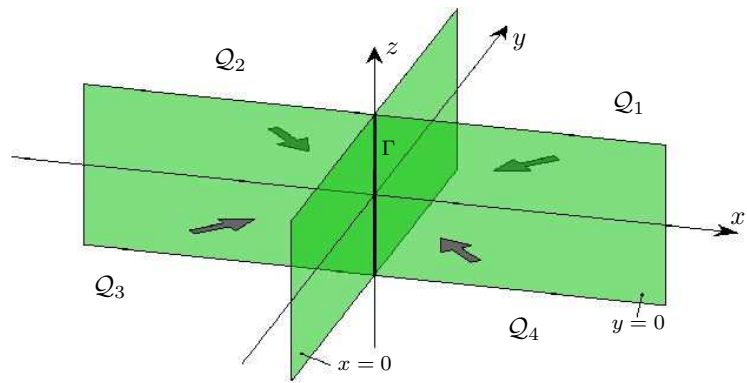}}
\put(9.68,2.12){\footnotesize $x$}
\put(6.78,4.5){\footnotesize $y$}
\put(5.09,4.38){\footnotesize $z$}
\put(8.1,3.63){\footnotesize ${\cal Q}_1$}
\put(2.8,4.2){\footnotesize ${\cal Q}_2$}
\put(1.2,1.18){\footnotesize ${\cal Q}_3$}
\put(6.7,.62){\footnotesize ${\cal Q}_4$}
\put(4.46,.18){\scriptsize $x=0$}
\put(8.1,.72){\scriptsize $y=0$}
\put(5.04,3.03){\scriptsize $\Gamma$}
\end{picture}
\caption{
A schematic of the phase space of three-dimensional piecewise-smooth system (\ref{eq:f}).
The coordinate planes $x=0$ and $y=0$ are discontinuity surfaces on which (\ref{eq:f}) is discontinuous.
\label{fig:schemIntSwMan}
} 
\end{center}
\end{figure}

\subsection{The convex hull and the canopy solution}
\label{sub:hullCanopy}

The convex hull of the components of (\ref{eq:f}) is set of all
\begin{equation}
\f f{\rm hull} = \sum_{i=1}^4 \gamma_i \f fi \;\qquad\mbox{where}\quad0\le\gamma_i\le1\quad\mbox{for each }i\;,
\label{eq:convexHull}
\end{equation}
subject to $\sum_{i=1}^4 \gamma_i = 1$.
Each multiplier $\gamma_i$ can be interpreted as the proportion of time spent in mode $i$.
When we seek sliding motion along $\Gamma$
we require that $\v f_{\rm hull}$ is tangent to the switching surfaces $x = 0$ and $y = 0$, and hence directs the motion to `slide' along their intersection.
This gives two additional restrictions on the $\gamma_i$
leading to a one parameter family of sliding solutions.
 
This ambiguity for two or more switches is well known \cite{Fi88,AlSe99,lopez}
and the most direct attempt to resolve it is the canopy combination \cite{AlSe98,Je14c}.
Alternatives include the moments solution \cite{Di15,DiDi15}
which selects differentiable orbits for the purposes of computation.
The canopy arises from the hull as follows. 

Let
\begin{equation}
\lambda_{14} = \gamma_1 + \gamma_4 \;, \qquad
\lambda_{12} = \gamma_1 + \gamma_2 \;,
\end{equation}
represent the proportion of time spent in $x > 0$ and $y > 0$ respectively.
If we treat the proportions $\gamma_i$ as independent probabilities, then
\begin{equation}\label{indep}
\gamma_1 = \lambda_{14} \lambda_{12} \;,
\end{equation}
which removes the ambiguity by providing an additional restriction on the values of the $\gamma_i$.
By using this to write each of the $\gamma_i$ in terms of $\lambda_{14}$ and $\lambda_{12}$ we obtain the canopy vector field
\begin{equation}
\f f{\rm canopy} =
\lambda_{14} \lambda_{12} \f f1 +
(1-\lambda_{14}) \lambda_{12} \f f2 +
(1-\lambda_{14}) (1-\lambda_{12}) \f f3 +
\lambda_{14} (1-\lambda_{12}) \f f4 \;,
\label{eq:canopy}
\end{equation}
where $\lambda_{14}$ and $\lambda_{12}$ are chosen such that $\f f{\rm canopy}$ is tangent to $x = 0$ and $y = 0$.
In (\ref{eq:canopy}) these two conditions give well-defined sliding solutions along $\Gamma$, with $\lambda_{12}$ and $\lambda_{14}$ fully determined.
If (\ref{indep}) does not hold then $\lambda_{12}$ and $\lambda_{14}$ are not fully determined by sliding motion. 



\subsection{Defining sliding motion through perturbations}
\label{sub:zeroPertLimit}

To define sliding motion on $\Gamma$ we consider perturbations
of (\ref{eq:f}) 
of size $\alpha$ in the $x$-direction and size $\beta$ in the $y$-direction.
We write these as
\begin{equation}
\alpha = \eps \cos \left( \frac{\pi \phi}{2} \right) \;, \qquad
\beta = \eps \sin \left( \frac{\pi \phi}{2} \right) \;,
\label{eq:alphabeta}
\end{equation} 
where $0 < \eps \ll 1$ and $0 < \phi < 1$ govern the overall size and ratio of the perturbations respectively.
The perturbations are defined precisely in later sections.
Here we explain how these perturbations are used to define sliding motion on $\Gamma$.


Orbits of a pertubation of (\ref{eq:f})
evolve in $\cO(\eps)$ (order $\eps$) neighbourhoods of attracting sections of $\Gamma$
(with the exception of the stochastic perturbation for which the probability of evolution far from $\Gamma$ vanishes as $\eps \to 0$).
The perturbed orbits switch modes frequently:
the difference between consecutive switching times is $\cO(\eps)$.
By taking the limit $\eps \to 0$, the rapid switching dynamics
can be averaged to define an effective sliding motion along $\Gamma$.

To see how this is achieved,
let $\v x(t)$ be an orbit of a pertubation of (\ref{eq:f}),
and suppose for simplicity that $\v x(0) = (0,0,z_0)$ for some $z_0 \in \mathbb{R}$.
Throughout a time interval $[0,T]$, where $\eps \ll T \ll 1$,
we have $\v f(\v x(t)) = \v f(\v x(0)) + \cO(T)$ (except in the stochastic perurbation for which this is false with vanishingly small probability).
Using this approximation, the identity
\begin{equation}
\v x(T) = \v x(0) + \int_0^T \v f(\v x(t)) \,dt \;,
\nonumber
\end{equation}
gives us
\begin{equation}
\v x(T) = \v x(0) + \sum_{i=1}^4 \gamma_i(z_0,\eps) T \left[ \f fi(\v x(0)) + \cO(T) \right] \;,
\label{eq:bxt}
\end{equation}
where $\gamma_i(z_0,\eps)$ is the fraction of $[0,T]$ for which $\v x(t)$ is in mode $i$, for $i = 1,\ldots,4$.
By using the limiting values $\gamma_i(z_0) = \lim_{\eps \to 0} \gamma_i(z_0,\eps)$
we can write
\begin{equation}
\v x(T) = \v x(0) + \f f{\rm slide}(z_0) T + \cO \left( T^2 \right) \;,
\label{eq:explicitAsymptoticSoln}
\end{equation}
where
\begin{equation}
\f f{\rm slide}(z) = \sum_{i=1}^4 \gamma_i(z) \f fi(0,0,z) \;.
\label{eq:fslide}
\end{equation}
Since $\v x(t)$ is by assumption contained in an $\cO(\eps)$ neighbourhood of $\Gamma$
and we have taken the limit $\eps \to 0$,
the vector field (\ref{eq:fslide}) must be tangent to both $x=0$ and $y=0$,
and so we can write $\f f{\rm slide}(z) = \left( 0,0,\ff h{\rm slide}(z) \right)$.
By taking $T \to 0$ in (\ref{eq:explicitAsymptoticSoln}) we
obtain $\dot{\v x}(t) = \f f{\rm slide}(\v x(t))$.
Therefore $\ff h{\rm slide}$ represents the sliding vector field defined by the zero-perturbation limit.

In the above derivation it is evident that we have the 
following three restrictions on the limiting values $\gamma_i(z)$:
\begin{equation}
\sum_{i=1}^4 \gamma_i(z) = 1 \;, \qquad
\sum_{i=1}^4 \gamma_i(z) f_i(0,0,z) = 0 \;, \qquad
\sum_{i=1}^4 \gamma_i(z) g_i(0,0,z) = 0 \;.
\label{eq:threeRestrictions}
\end{equation}
These restrictions leave us with a one-parameter family of values possible for the $\gamma_i(z)$.
The idea is that the $\gamma_i(z)$ are determined from the rapid switching dynamics of the 
pertubation of (\ref{eq:f}).

To understand these dynamics it suffices to replace each $\f fi$ in (\ref{eq:f})
with its value at $\v x = (0,0,z)$ because we take the limit $T \to 0$.
This reduces (\ref{eq:f}) to the piecewise-constant system,
\begin{equation}
(\dot{x},\dot{y},\dot{z}) =
\begin{cases}
(\ff a1,\ff b1,\ff c1) \;, & (x,y) \in {\cal Q}_1 \\[-1.6mm]
\hspace{8mm} \vdots \\[-1.6mm]
(\ff a4,\ff b4,\ff c4) \;, & (x,y) \in {\cal Q}_4
\end{cases} \;,
\label{eq:f3dPWC}
\end{equation}
where we introduce the constant vector $\f fi = (\ff ai,\ff bi,\ff ci)$, for each $i$.
The $x$ and $y$ components of (\ref{eq:f3dPWC}) are decoupled from $z$
and hence we can restrict our attention to the two-dimensional system
\begin{equation}
(\dot{x},\dot{y}) =
\begin{cases}
(\ff a1,\ff b1) \;, & (x,y) \in {\cal Q}_1 \\[-1.6mm]
\hspace{5mm} \vdots \\[-1.6mm]
(\ff a4,\ff b4) \;, & (x,y) \in {\cal Q}_4
\end{cases} \;.
\label{eq:f2dPWC}
\end{equation}

Furthermore, to analyse pertubations of (\ref{eq:f3dPWC}) it is helpful to scale $\v x$ and $t$ by $\eps$.
With this scaling, perturbations are $\cO(1)$,
the time interval over which the $\gamma_i$ are determined is $\left[ 0 ,\, \frac{T}{\eps} \right]$,
but the components of (\ref{eq:f3dPWC}) are unchanged because they are constant.
Assuming $\frac{T}{\eps} \to \infty$ as $\eps$ and $T$ are taken to zero,
we conclude that {\em the $\gamma_i$ are to be determined from attractors of
the pertubation of (\ref{eq:f2dPWC}) using $\eps = 1$}.
The induced sliding solution is then the solution to the vector field
\begin{equation}
h_{\rm slide} = \sum_{i=1}^4 \gamma_i c_i \;.
\label{eq:hSlide}
\end{equation}

\subsection{Jitter: a simulation}
\label{sub:simulation}

To illustrate the extreme sensitivity typical in $\ff h{\rm slide}$,
we consider a one-parameter family of vector fields
obtained by linearly interpolating two random piecewise-constant vector fields.
A similar approach is used in \cite{AlSe99}.

For each $i$ we write
\begin{equation}
\begin{split}
\ff fi(x,y,z) &= (1-z) a_{i,0} + z a_{i,1} \;, \\
\ff gi(x,y,z) &= (1-z) b_{i,0} + z b_{i,1} \;, \\
\ff hi(x,y,z) &= (1-z) c_{i,0} + z c_{i,1} \;,
\end{split}
\label{eq:fghLinearInt}
\end{equation}
where
\begin{equation}
\begin{aligned}
a_{1,0} &= -0.5822 \;, & a_{2,0} &= 0.5408 \;, & a_{3,0} &= 0.8700 \;, & a_{4,0} &= -0.2647 \;, \\
b_{1,0} &= -0.3180 \;, & b_{2,0} &= -0.1192 \;, & b_{3,0} &= 0.9399 \;, & b_{4,0} &= 0.6456 \;, \\
c_{1,0} &= -0.0410 \;, & c_{2,0} &= 0.2788 \;, & c_{3,0} &= 0.0896 \;, & c_{4,0} &= 0.2948 \;,
\end{aligned}
\label{eq:abc0}
\end{equation}
and
\begin{equation}
\begin{aligned}
a_{1,1} &= -0.5438 \;, & a_{2,1} &= 0.7211 \;, & a_{3,1} &= 0.5225 \;, & a_{4,1} &= -0.9937 \;, \\
b_{1,1} &= -0.2186 \;, & b_{2,1} &= -0.1057 \;, & b_{3,1} &= 0.1097 \;, & b_{4,1} &= 0.0636 \;, \\
c_{1,1} &= -0.1908 \;, & c_{2,1} &= -0.1032 \;, & c_{3,1} &= -0.2682 \;, & c_{4,1} &= 0.5272 \;.
\end{aligned}
\label{eq:abc1}
\end{equation}
The values $a_{i,j}$ and $b_{i,j}$ were obtained by sampling from uniform distributions on $[0,1]$ and $[-1,0]$
(as appropriate to get the desired signs) and the $c_{i,j}$
were obtained by sampling from a uniform distribution on $[-1,1]$.
The values were rounded to four decimal places (for easy reproducibility)
and we consider $z \in [0,1]$ with which each $\f fi$ is directed inwards.

The restrictions (\ref{eq:threeRestrictions}) allow for only a finite range of values of $\ff h{\rm slide}$.
This range is indicated in Fig.~\ref{fig:allBifDiag_z21} as an unshaded strip
and the canopy solution (\ref{eq:canopy}) is shown as a dashed line.
Fig.~\ref{fig:allBifDiag_z21} shows $\ff h{\rm slide}$ for the four different perturbations considered below
(using $\phi = \frac{1}{2}$ in each case).
With the exception of perturbation by noise,
the graphs of $\ff h{\rm slide}$ are highly irregular in places; 
this is not numerical error, $\ff h{\rm slide}$ actually does this!.
We refer to this phenomenon as {\em jitter}.

\begin{figure}[t!]
\begin{center}
\setlength{\unitlength}{1cm}
\begin{picture}(15,7.5)
\put(0,0){\includegraphics[height=7.5cm]{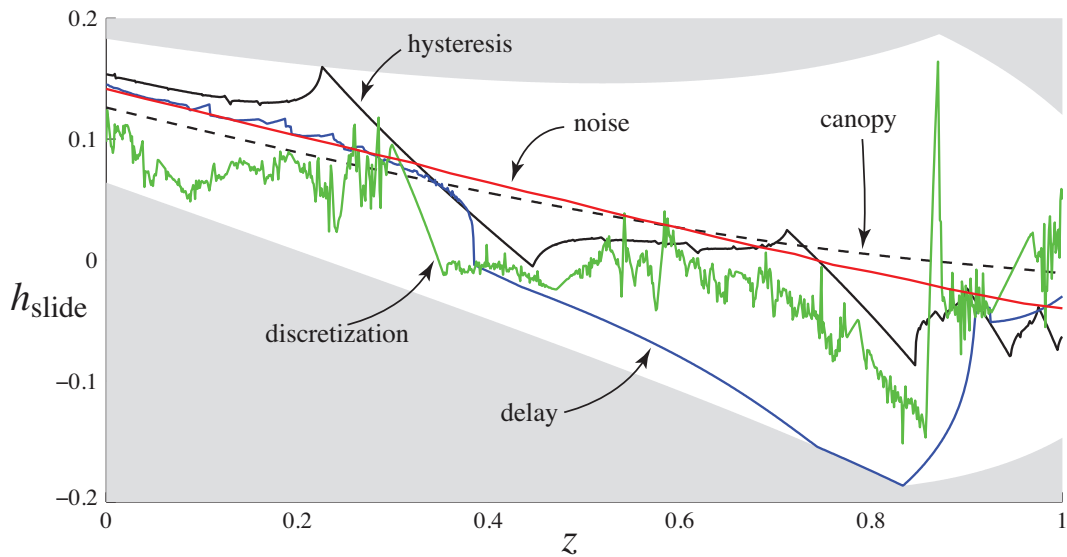}}
\end{picture}
\caption{
Different derivations of the sliding vector field $h_{\rm slide}$ on $\Gamma$
(the intersection of the discontinuity surfaces $x=0$ and $y=0$) at different values of $z$
for the piecewise-smooth system (\ref{eq:f}) with (\ref{eq:fghLinearInt})-(\ref{eq:abc1}), as discussed in the text.
The unshaded region is the convex hull, and the curves show: the canopy (dashed), hysteresis (black), white noise (red), time-discretization (green), time-delay (blue); as labelled [color online]. 
\label{fig:allBifDiag_z21}
}
\end{center}
\end{figure}




\section{Hysteresis}
\label{sec:hysteresis}
\setcounter{equation}{0}

Here we introduce hysteresis to the two switches in (\ref{eq:f})
by requiring that switching occurs with $\alpha$ and $\beta$ spatial delays about $x=0$ and $y=0$ respectively.
In full this means applying the following rules for changing the mode
of an orbit when $x(t) = \pm \alpha$ or $y(t) = \pm \beta$:
\begin{equation}
\begin{aligned}
& \raisebox{.4mm}{\tiny $\bullet$} {\rm ~~at~} x = \alpha, {\rm ~mode~} 2 \to 1 {\rm ~and~} 3 \to 4, \\
& \raisebox{.4mm}{\tiny $\bullet$} {\rm ~~at~} x = -\alpha, {\rm ~mode~} 1 \to 2 {\rm ~and~} 4 \to 3, \\
& \raisebox{.4mm}{\tiny $\bullet$} {\rm ~~at~} y = \beta, {\rm ~mode~} 3 \to 2 {\rm ~and~} 4 \to 1, \\
& \raisebox{.4mm}{\tiny $\bullet$} {\rm ~~at~} y = -\beta, {\rm ~mode~} 2 \to 3 {\rm ~and~} 1 \to 4.
\end{aligned}
\label{eq:hysteresisRules}
\end{equation}
It is possible for two switches to coincide unambiguously, for example from mode $i=1$ at $x=-\alpha$ and $y=-\beta$ with both $x$ and $y$ decreasing, the mode switches to $i=3$ (either via $i:1\mapsto2\mapsto3$ or $i:1\mapsto4\mapsto3$).

For the piecewise-constant vector field (\ref{eq:f3dPWC}) perturbed by hysteresis:
(i) the structure of the attractor dynamics is determined by the four angles ${\rm atan2}(b_i,a_i)$,
(ii) the relative times spent in each quadrant are determined by the eight values of $a_i$ and $b_i$,
(iii) the zero-hysteresis sliding solution is determined by the $12$ values of $a_i$, $b_i$ and $c_i$.

\subsection{Chatterbox dynamics}
\label{sub:chatterbox}

Here we summarise the theoretical results of Alexander and Seidman \cite{AlSe99}
for the hysteretic system (\ref{eq:f2dPWC}) with (\ref{eq:hysteresisRules}).

In the case that each $\f fi$ is directed inwards,
orbits of (\ref{eq:f2dPWC}) with (\ref{eq:hysteresisRules}) cannot escape the rectangle
\begin{equation}
\Omega = \left\{ (x,y) ~\big|~ |x| \le \alpha ,\, |y| \le \beta \right\} \;.
\label{eq:chatterbox}
\end{equation}
Following \cite{AlSe99} we refer to $\Omega$ as a {\em chatterbox}.

\begin{figure}[b!]
\begin{center}
\setlength{\unitlength}{1cm}
\begin{picture}(15.17,5.8)
\put(0,0){\includegraphics[height=5.5cm]{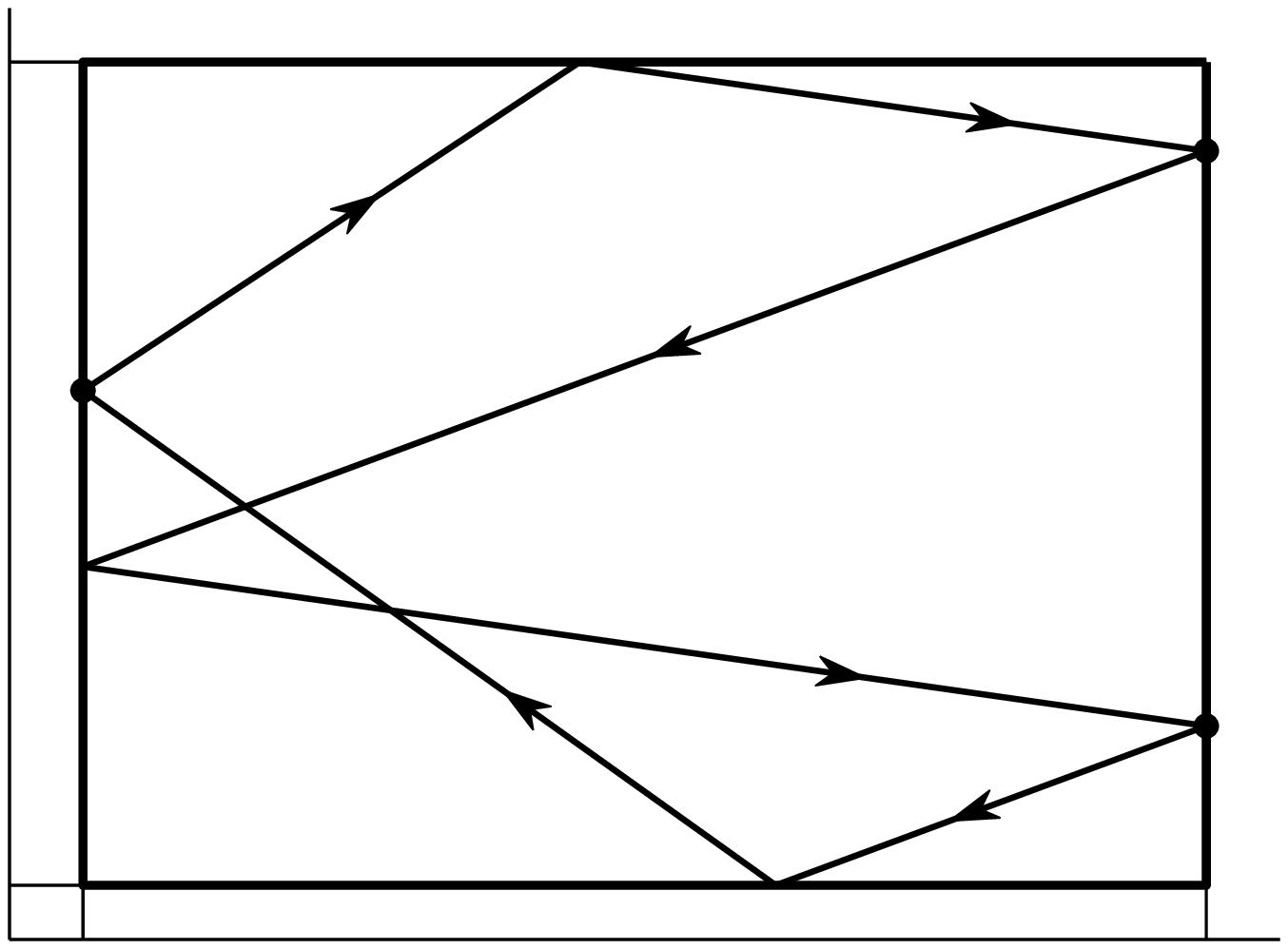}}
\put(7.83,0){\includegraphics[height=5.5cm]{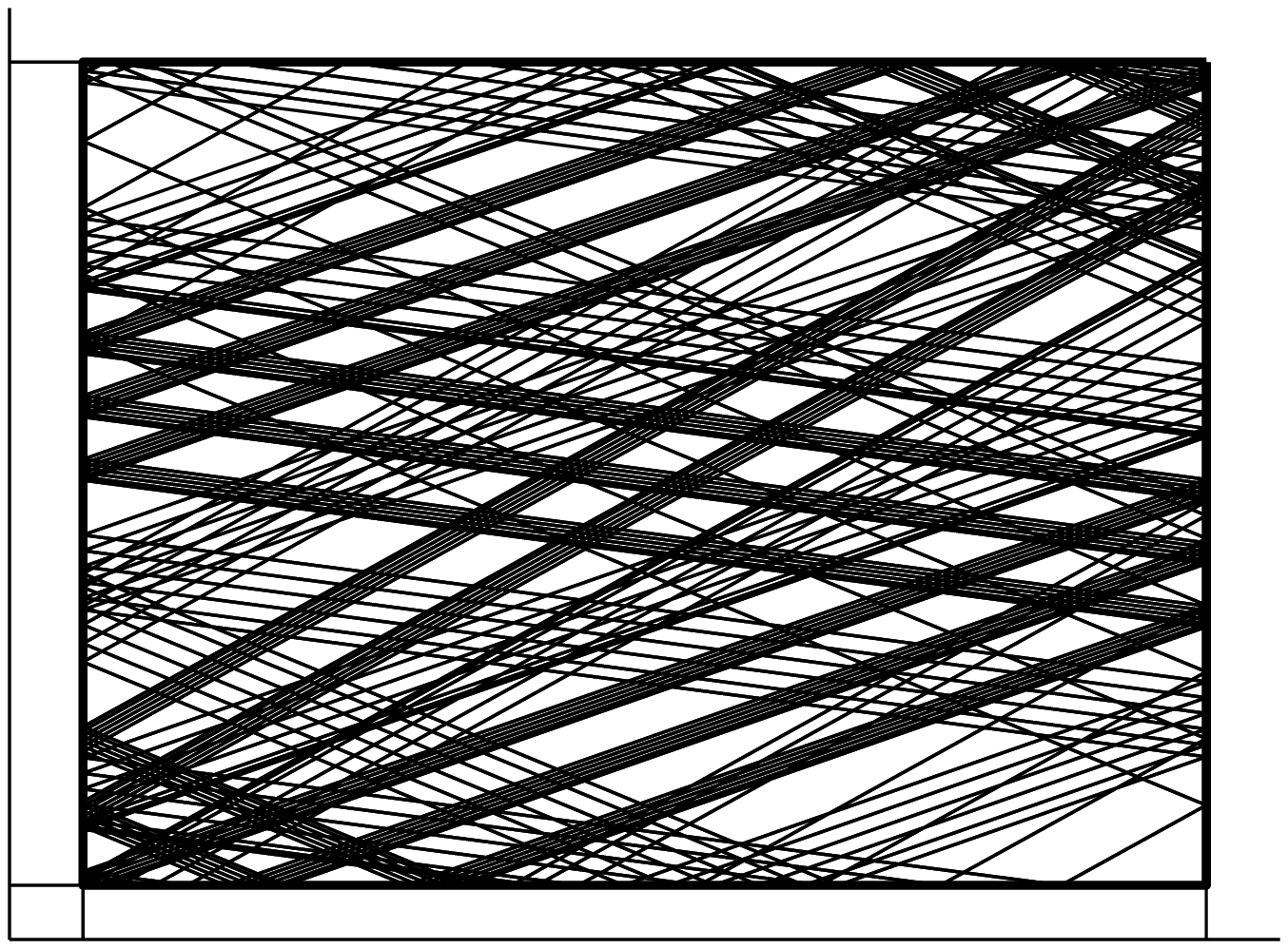}}
\put(.9,5.5){\large \sf \bfseries A}
\put(3.8,0){$x$}
\put(0,3.1){$y$}
\put(.7,.28){\small $-\alpha$}
\put(6.75,.28){\small $\alpha$}
\put(-.02,.74){\small $-\beta$}
\put(.28,5.07){\small $\beta$}
\put(3.92,3.37){\scriptsize $1$}
\put(5,1.68){\scriptsize $2$}
\put(5.44,1.27){\scriptsize $1$}
\put(3.1,1.57){\scriptsize $4$}
\put(2.4,4.56){\scriptsize $3$}
\put(5.77,4.6){\scriptsize $2$}
\put(7,4.66){\scriptsize $\eta \approx 0.97$}
\put(7,1.63){\scriptsize $\eta \approx 0.80$}
\put(1.17,3.4){\scriptsize $\eta \approx 0.35$}
\put(8.73,5.5){\large \sf \bfseries B}
\put(11.63,0){$x$}
\put(7.83,3.1){$y$}
\put(8.53,.28){\small $-\alpha$}
\put(14.58,.28){\small $\alpha$}
\put(7.81,.74){\small $-\beta$}
\put(8.11,5.07){\small $\beta$}
\end{picture}
\caption{
Attractors in the chatterbox $\Omega$ (\ref{eq:chatterbox})
of the two-dimensional piecewise-constant system (\ref{eq:f2dPWC}) perturbed by hysteresis (\ref{eq:hysteresisRules})
for the random example (\ref{eq:fghLinearInt})-(\ref{eq:abc1}).
In panel A, $z = 0.3$; in panel B, $z = 0.5$.
Both panels use $\phi = 0.5$ and $\eps = 1$ in (\ref{eq:alphabeta}) --- the size of $\Omega$.
More precisely, each panel shows part of a numerically computed forward orbit with transient dynamics removed.
In panel A, we have labelled different parts of the orbit by its mode
and alternate intersections with the boundary of $\Omega$
by the value of $\eta$ (to two decimal places) for the circle map defined in \sect{sub:circleMap}.
\label{fig:chatterbox21}
} 
\end{center}
\end{figure}

The system (\ref{eq:f2dPWC}) with (\ref{eq:hysteresisRules}) can be reformulated
as a doubly-periodic vector field on the $(x,y)$-plane
by tiling the $(x,y)$-plane with copies of $\Omega$
in a such way that whenever an orbit reaches the boundary of one $\Omega$ it passes continuously through
into an adjacent $\Omega$.
This is helpful because we can then apply general results established for
two-dimensional doubly-periodic vector fields \cite{Ar88}.
Firstly, we can say that the system has a unique rotation number
(although this does not determine the sliding solution which depends on the
proportions of time spent in each mode).
Secondly, we can say that the rotation number is almost always rational
in the sense that in the space of all systems (\ref{eq:f2dPWC}) with (\ref{eq:hysteresisRules}),
the subset having irrational rotation number is of measure zero.
Rational rotation numbers correspond to periodic orbits.
Alexander and Seidman show that if (\ref{eq:f2dPWC}) with (\ref{eq:hysteresisRules}) has a rational rotation number
then there exist exactly two periodic orbits.
One of these is attracting and attracts almost all points in $\Omega$, the other is unstable.
Therefore almost every $(x_0,y_0) \in \Omega$ has the same $\omega$-limit set
and hence leads to the same values of $\gamma_i$.
In this sense the zero-hysteresis sliding solution is unique.
In addition, Alexander and Seidman show that this solution
is a Lipschitz function of the components of (\ref{eq:f2dPWC}).

Fig.~\ref{fig:chatterbox21} shows attractors of (\ref{eq:f}) with (\ref{eq:fghLinearInt})-(\ref{eq:abc1}).
In panel A the attractor has six switches per period.
In panel B the attractor is either quasi-periodic
or has such a high period that we were unable to detect it.

\subsection{Circle map}
\label{sub:circleMap}

The dynamics within $\Omega$ can be captured by a map between consecutive switching points.
However, if we ignore switches at corners (which constitute a measure-zero subset),
when an orbit switches its mode changes from even to odd or vice-versa.
For this reason it is simpler to work with a map from the $n^{\rm th}$ switching point
to the $(n+2)^{\rm th}$ switching point.

Let $(x,y)$ be any point on the boundary of $\Omega$ corresponding to an orbit $\v x(t)$
that has just undergone a switch.
By assuming that the mode of $\v x(t)$ at this point is odd,
the mode is completely determined by the point $(x,y)$.
Specifically, if $(x,y)$ lies on either the top or right edge of $\Omega$ then $\v x(t)$ is in mode $1$
and if $(x,y)$ lies on either the bottom or left edge of $\Omega$ then $\v x(t)$ is in mode $3$.

We map the boundary of $\Omega$ to a circle $\mathbb{S}^1$ by the continuous function
\begin{equation}
\eta = B(x,y) = \begin{cases}
\frac{\alpha - x}{8 \alpha} \;, & y = \beta \\
\frac{3 \beta - y}{8 \beta} \;, & x = -\alpha \\
\frac{5 \alpha + x}{8 \alpha} \;, & y = -\beta \\
\frac{7 \beta + y}{8 \beta} \;, & x = \alpha
\end{cases} \;.
\label{eq:B}
\end{equation}
With this definition the corners of $\Omega$ map to integer multiples of $\frac{1}{4}$
as indicated in Fig.~\ref{Fig:circle}-B.

Each $\eta \in \mathbb{S}^1$ corresponds to a point $(x,y)$ on the boundary of $\Omega$
in either mode $1$ or mode $3$, as determined by the side of $\Omega$ to which $(x,y)$ belongs.
For any $\eta \in \mathbb{S}^1$, we let $q(\eta) \in \mathbb{S}^1$ correspond to the location
of the forward orbit of $(x,y)$ immediately after its second switch.
Since orbits in $\Omega$ follow lines, and (\ref{eq:B}) is piecewise-linear, the map $q$ is piecewise-linear.
Moreover, it is a continuous, invertible, degree-one circle map that generically involves eight pieces.

To describe $q$ more precisely, let $u_0 < \cdots < u_7$ be the ordered list of the values
\begin{equation}
\textstyle \left\{ 0, \frac{1}{4}, \frac{1}{2}, \frac{3}{4},
q^{-1}(0), q^{-1}(\frac{1}{4}), q^{-1}(\frac{1}{2}), q^{-1}(\frac{3}{4}) \right\} \;.
\nonumber
\end{equation}
and let $v_i = q(u_i)$ denote their images under $q$.
The graph of $q$ is then given by simply connecting each $(u_i,v_i)$ to
$\left( u_{(i+1) {\rm \,mod\,} 8},v_{(i+1) {\rm \,mod\,} 8} \right)$ by a line segment,
taking care to appropriately deal with the topology of $\mathbb{S}^1$, see Fig.~\ref{Fig:circle}-A.

\begin{figure}[h!]
	\begin{center}
		\includegraphics[width=0.8\textwidth]{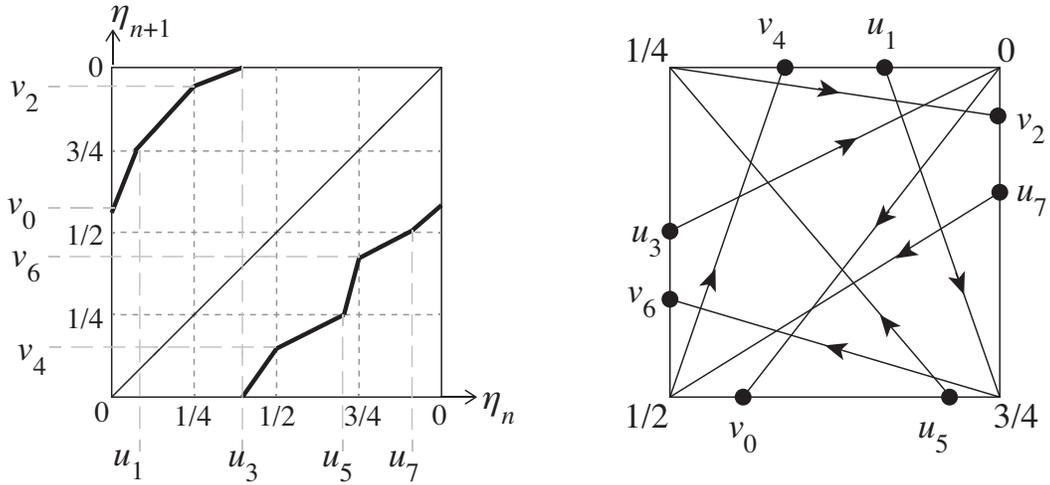}
      	\caption{Geometry of the second return map $q$. An example map is shown on the left: the points $u_3,u_5,u_7,u_1$, map to $\eta=0,\frac14,\frac12,\frac34$, which in turn map to $v_3,v_2,v_4,v_6$. On the right, the second return maps are shown on the chatter box in $\eta$ coordinates (where the corners are $\eta=0,\frac14,\frac12,\frac34$. (Note each lines shows the result of two excursions across the box, {\it not} one, so they do not lie along the vector fields). }
      	\label{Fig:circle}
	\end{center}
\end{figure}	


\begin{figure}[h!]
\begin{center}
\setlength{\unitlength}{1cm}
\begin{picture}(15.17,5.8)
\put(0,0){\includegraphics[height=5.5cm]{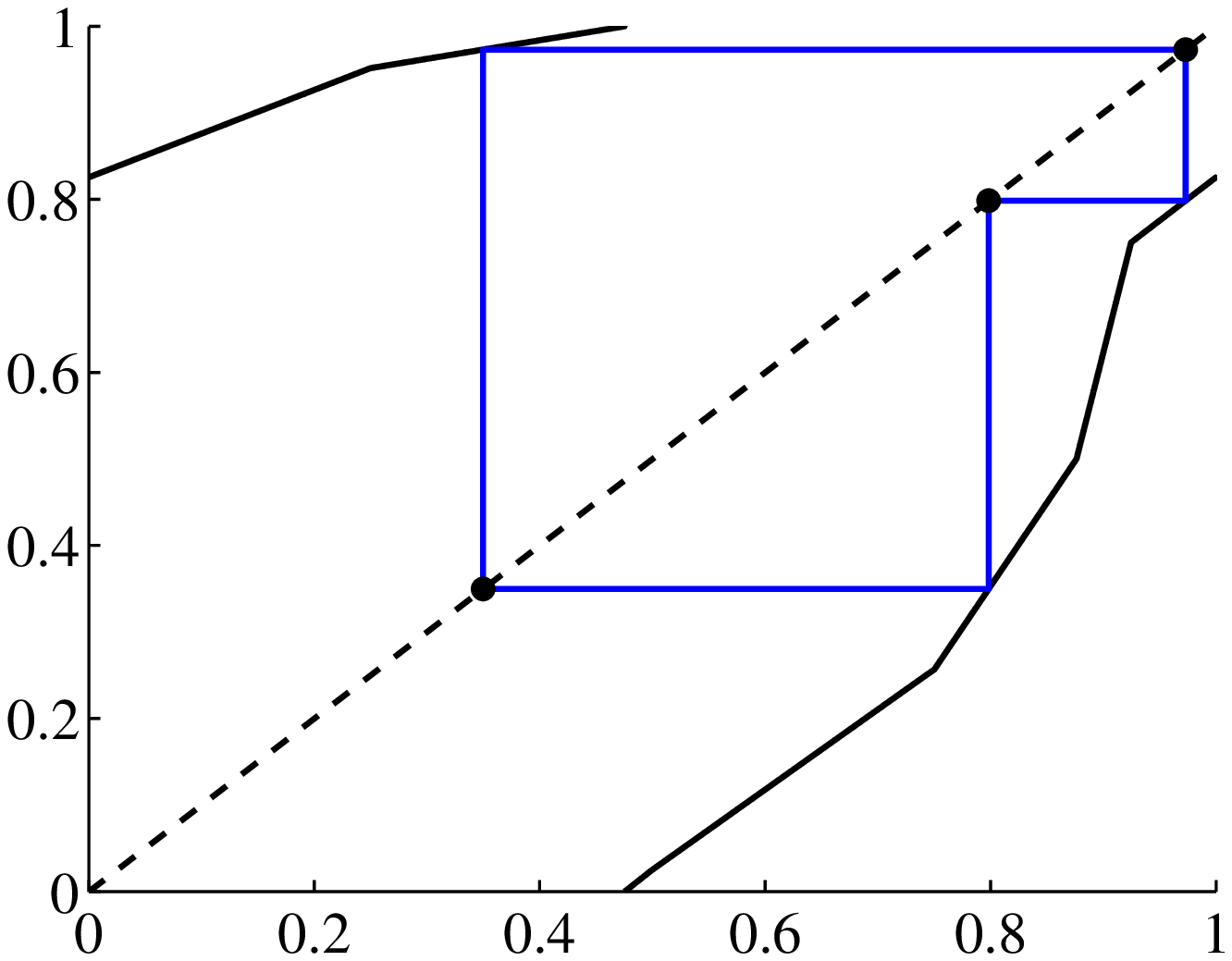}}
\put(7.83,0){\includegraphics[height=5.5cm]{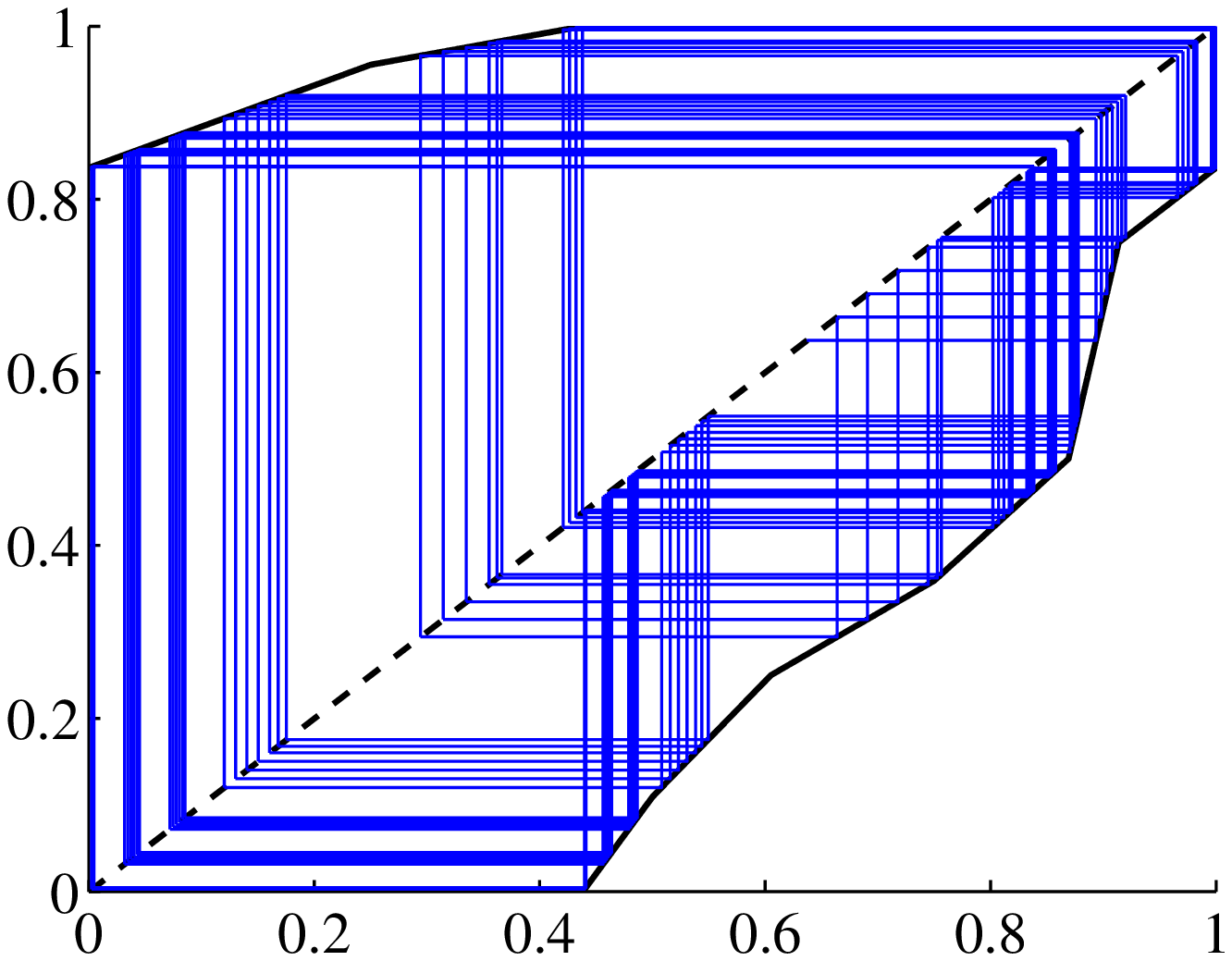}}
\put(1.3,5.5){\large \sf \bfseries A}
\put(4.05,0){$\eta$}
\put(0,2.9){$q(\eta)$}
\put(9.13,5.5){\large \sf \bfseries B}
\put(11.88,0){$\eta$}
\put(7.83,2.9){$q(\eta)$}
\end{picture}
\caption{
Attractors of the circle map $q$
for the system (\ref{eq:f}) with (\ref{eq:fghLinearInt})-(\ref{eq:abc1}) and $\phi = 0.5$.
In panel A, $z = 0.3$; in panel B, $z = 0.5$.
These correspond to the attractors in $\Omega$ shown in Fig.~\ref{fig:chatterbox21}.
\label{fig:circleMap21}
} 
\end{center}
\end{figure}


Fig.~\ref{fig:circleMap21} illustrates the circle map for the system
(\ref{eq:f}) with (\ref{eq:fghLinearInt})-(\ref{eq:abc1}).
These show the attractors of Fig.~\ref{fig:chatterbox21} in the context of the circle map.
Fig.~\ref{fig:modeLockingRegions21} shows how the rotation number and period of the attractor in $\Omega$
varies over the entire range of $\phi$ and $z$ values.
There are open regions for which the rotation number is constant --- these are {\em mode-locking regions}.
We observe that the mode-locking regions have points of zero width.
This phenomenon has been described for the sawtooth map \cite{YaHa87}
and piecewise-linear maps on $\mathbb{R}^n$ \cite{SiMe09,Si15c},
but to our knowledge has not previously been detected for piecewise-linear maps
involving more than two different derivatives.

\begin{figure}[h!]
\begin{center}
\setlength{\unitlength}{1cm}
\begin{picture}(15,6)
\put(0,0){\includegraphics[height=6cm]{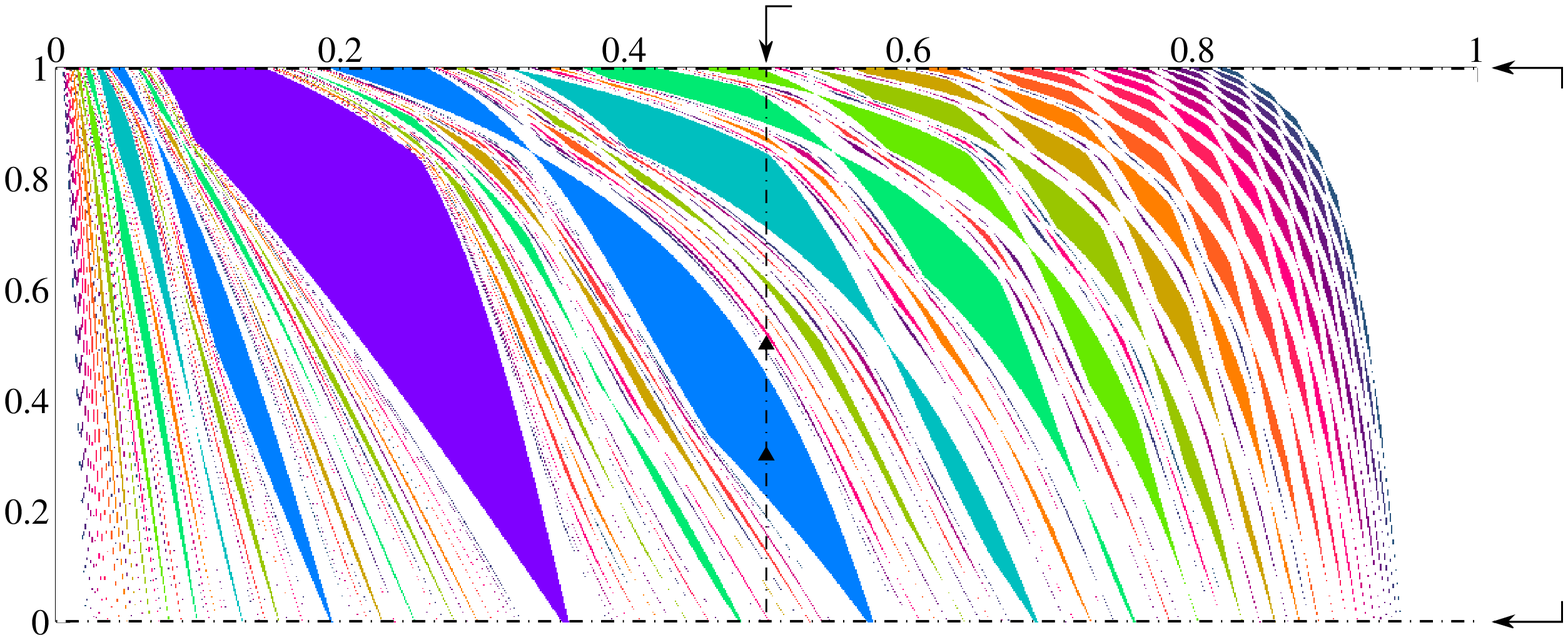}}
\put(4.8,5.75){$\phi$}
\put(.2,2.8){$z$}
\put(2.32,.18){\scriptsize $\frac{1}{5}$}
\put(2.74,.18){\scriptsize $\frac{1}{4}$}
\put(3.48,.18){\scriptsize $\frac{1}{3}$}
\put(4.21,.18){\scriptsize $\frac{2}{5}$}
\put(5.55,.18){\scriptsize $\frac{1}{2}$}
\put(7.06,.18){\scriptsize $\frac{3}{5}$}
\put(8.20,.18){\scriptsize $\frac{2}{3}$}
\put(9.64,.18){\scriptsize $\frac{3}{4}$}
\put(10.5,.18){\scriptsize $\frac{4}{5}$}
\put(7.73,5.8){\scriptsize Fig.~\ref{fig:allBifDiag_z21}}
\put(13.9,.76){\scriptsize Fig.~\ref{fig:hystBifDiag_phi21}-A}
\put(13.9,4.92){\scriptsize Fig.~\ref{fig:hystBifDiag_phi21}-B}
\put(7.53,1.88){\scriptsize \colorbox{white}{Fig.~\ref{fig:circleMap21}-A}}
\put(7.53,2.84){\scriptsize \colorbox{white}{Fig.~\ref{fig:circleMap21}-B}}
\end{picture}
\caption{
Mode-locking regions of the circle map $q$
for the system (\ref{eq:f}) with (\ref{eq:fghLinearInt})-(\ref{eq:abc1}).
Different colors correspond to different periods $n$ up to $n = 20$;
uncolored regions correspond to a period $n > 20$ or a quasiperiodic orbit [color online].
At the bottom of the figure we have indicated the rotation numbers $\frac{m}{n}$ of mode-locking regions with $n \le 5$.
Several other figures correspond to cross-sections or particular points in this figure as indicated.
\label{fig:modeLockingRegions21}
} 
\end{center}
\end{figure}

\subsection{The zero-hysteresis solution}
\label{sub:zeroHystSoln}

As discussed in \sect{sub:zeroPertLimit},
the value of each $\gamma_i$ in the sliding vector field (\ref{eq:hSlide})
is defined as the fraction of time that an orbit following the attractor spends in mode $i$.
For the attractor shown in Fig.~\ref{fig:chatterbox21}-A, for example,
we have $\gamma_1 = 0.35$, $\gamma_2 = 0.38$, $\gamma_3 = 0.08$ and $\gamma_4 = 0.19$, to two decimal places.
Then $h_{\rm slide} = 0.10$ (to two decimal places).
This corresponds to $z = 0.3$ in Fig.~\ref{fig:allBifDiag_z21}.

As shown in Fig.~\ref{fig:allBifDiag_z21}, the variation in the zero-hysteresis value of $h_{\rm slide}$
is rather erratic for some ranges of values of $z$ but smooth for other ranges.
The smooth ranges correspond to attractors with a fixed rotation number.
For instance between $z \approx 0.23$ to $z \approx 0.45$ the rotation number is $\frac{2}{3}$,
as shown in Fig.~\ref{fig:modeLockingRegions21}, and $h_{\rm slide}$ decreases smoothly with the value of $z$.


\begin{figure}[h!]
\begin{center}
\setlength{\unitlength}{1cm}
\begin{picture}(15.17,5.8)
\put(0,0){\includegraphics[height=5.5cm]{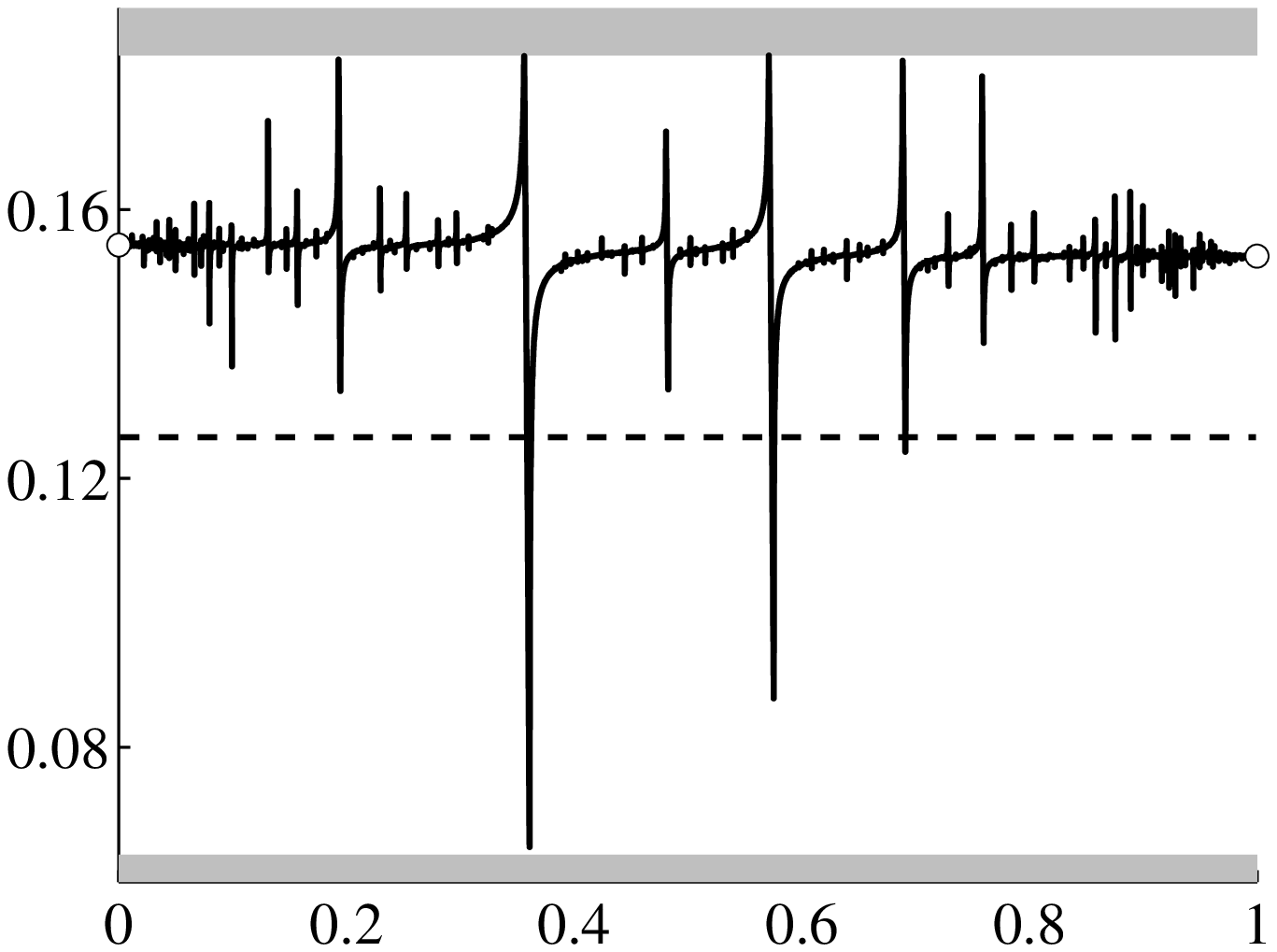}}
\put(7.83,0){\includegraphics[height=5.5cm]{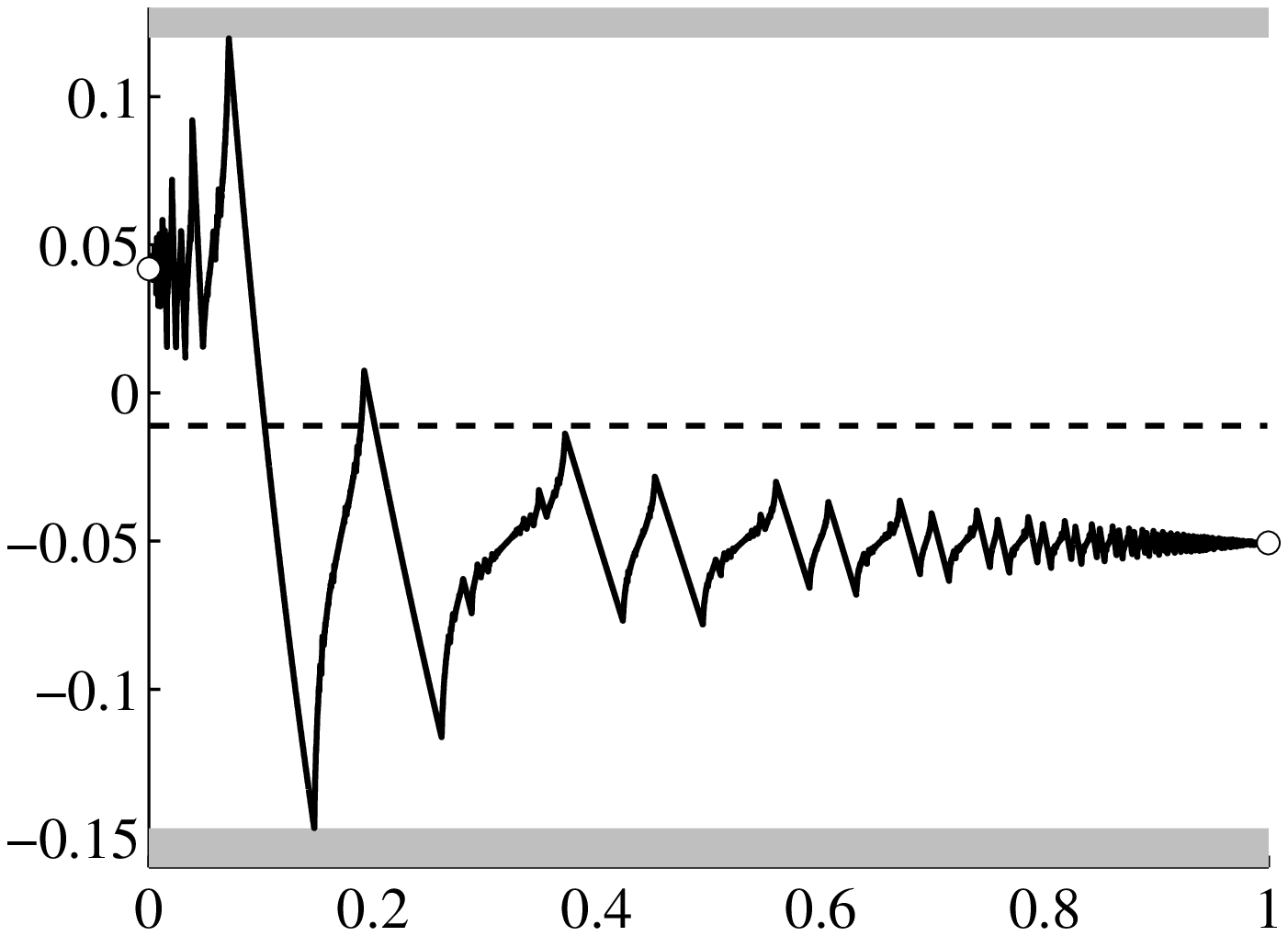}}
\put(1.2,5.5){\large \sf \bfseries A}
\put(4.02,0){$\phi$}
\put(0,3.47){$h_{\rm slide}$}
\put(9.03,5.5){\large \sf \bfseries B}
\put(11.85,0){$\phi$}
\put(7.72,3.15){$h_{\rm slide}$}
\end{picture}
\caption{
The sliding vector field $h_{\rm slide}$ as a function of $\phi$
for (\ref{eq:f}) perturbed by hysteresis (\ref{eq:hysteresisRules}) with (\ref{eq:fghLinearInt})-(\ref{eq:abc1})
and $z = 0$ in panel A and $z = 1$ in panel B.
The values for the $\phi = 0$ and $\phi = 1$ limits, as given by Proposition \ref{pr:phiLimits}, are indicated by circles. The canopy solution is indicated by a dashed line, and the hull corresponds to the entire range of $h_{\rm slide}$ values shown.
\label{fig:hystBifDiag_phi21}
} 
\end{center}
\end{figure}

Sensitivity in the value of $h_{\rm slide}$ also occurs as we vary the value of $\phi$, Fig.~\ref{fig:hystBifDiag_phi21}.
These two plots correspond to horizontal cross-sections of Fig.~\ref{fig:modeLockingRegions21}
and so again we can see that while the rotation number is constant the value of $h_{\rm slide}$ varies smoothly.
However, as we go from $\phi = 0$ to $\phi = 1$ the rotation number varies all the way across its two most extreme values ($0$ to $1$).
Consequently intervals of constant rotation number are relatively small
and for this reason $h_{\rm slide}$ is an extremely erratic function of $\phi$.

The value of $h_{\rm slide}$ can be given explicitly
for the limits $\phi \to 0$ and $\phi \to 1$.
With $\phi \approx 0$ or $\phi \approx 1$, the chatterbox $\Omega$ is a narrow rectangle.
As orbits in $\Omega$ travel between the short sides of $\Omega$
they switch many times between two modes.
During this time the motion is well-approximated by Filippov's solution for a single discontinuity surface.
We can then average the two Filippov solutions to obtain $h_{\rm slide}$.

To state the limiting values of $h_{\rm slide}$ we require some additional notation.
We let
\begin{equation}
\lambda_{14} = \frac{b_4}{b_4 - b_1} \;, \qquad
\lambda_{23} = \frac{b_3}{b_3 - b_2} \;, \qquad
\lambda_{12} = \frac{a_2}{a_2 - a_1} \;, \qquad
\lambda_{43} = \frac{a_3}{a_3 - a_4} \;,
\end{equation}
denote the switching multipliers, see (\ref{eq:convexCombination}),
for switching between two adjacent modes.
For each of the four adjacent pairs of modes, $ij$, we let
\begin{equation}
\f f{ij} = \lambda_{ij} \f fi + (1 - \lambda_{ij}) \f fj \;,
\end{equation}
denote Filippov's sliding solution,
and write $\f f{ij} = (a_{ij},b_{ij},c_{ij})$.
We also let
\begin{equation}
\lambda_{14,23} = \frac{a_{23}}{a_{23} - a_{14}} \;, \qquad
\lambda_{12,43} = \frac{b_{43}}{b_{43} - b_{12}} \;,
\end{equation}
in order to average the corresponding pairs of $\f f{ij}$.

The following result provides the limiting values of $\gamma_i$
from which $h_{\rm slide}$ is given by (\ref{eq:hSlide}).

\begin{proposition}
At any point on $\Gamma$ at which each $\f fi$ of (\ref{eq:f}) is directed inwards,
the zero-hysteresis sliding solution yields the following limiting values of $\gamma_i$.
In the limit $\phi \to 0$,
\begin{equation}
\gamma_1 = \lambda_{14,23} \lambda_{14} \;, \quad
\gamma_2 = (1-\lambda_{14,23}) \lambda_{23} \;, \quad
\gamma_3 = (1-\lambda_{14,23}) (1-\lambda_{23}) \;, \quad
\gamma_4 = \lambda_{14,23} (1-\lambda_{14}) \;,
\label{eq:gammaphizero}
\end{equation}
and in the limit $\phi \to 1$,
\begin{equation}
\gamma_1 = \lambda_{12,43} \lambda_{12} \;, \quad
\gamma_2 = \lambda_{12,43} (1-\lambda_{12}) \;, \quad
\gamma_3 = (1-\lambda_{12,43}) (1-\lambda_{43}) \;, \quad
\gamma_4 = (1-\lambda_{12,43}) \lambda_{43} \;.
\label{eq:gammaphione}
\end{equation}
\label{pr:phiLimits}
\end{proposition}




\begin{proof}
Here we prove the result for $\phi = 0$; the result for $\phi = 1$ follows by symmetry.
By the results of \cite{AlSe99}, see \sect{sub:chatterbox},
there exists a unique attractor in the chatterbox $\Omega$
(or an orbit that densely fills $\Omega$ and plays the role of an attractor).
For each $i$, $\gamma_i$ is the fraction of time that the attractor spends in mode $i$.

First consider an orbit in the attractor as it travels from
any point on the right boundary of $\Omega$, $x = \alpha$, until reaching the left boundary, $x = -\alpha$.
During this time the orbit switches between modes $1$ and $4$.
With $\phi \approx 0$, the number of switches is $\cO \left( \frac{1}{\phi} \right)$.
Thus the fraction of time spent in mode $1$ is $\tilde{\lambda}_{14} = \lambda_{14} + \cO(\phi)$,
and the fraction of time spent in mode $4$ is $1-\tilde{\lambda}_{14}$.
Also the time taken to travel from $x = \alpha$ to $x = -\alpha$ is
\begin{equation}
T_{14} = \frac{-2 \alpha}{\tilde{\lambda}_{14} f_1 + (1-\tilde{\lambda}_{14}) f_4} \;.
\label{eq:T14}
\end{equation}

Upon reaching $x = -\alpha$, the orbit subsequently travels back to $x = \alpha$ switching between modes $2$ and $3$.
The fraction of time spent in mode $2$ is $\tilde{\lambda}_{23} = \lambda_{23} + \cO(\phi)$,
and the fraction of time spent in mode $3$ is $1-\tilde{\lambda}_{23}$.
Also the time taken to travel from $x = -\alpha$ to $x = \alpha$ is
\begin{equation}
T_{23} = \frac{2 \alpha}{\tilde{\lambda}_{23} f_2 + (1-\tilde{\lambda}_{23}) f_3} \;.
\label{eq:T23}
\end{equation}

By combining these observations, we see that as the orbit travels from $x = \alpha$ until it next arrives at this boundary,
the fraction of time spent in mode $1$, for instance, is
$\frac{T_{14} \tilde{\lambda}_{14}}{T_{14} + T_{23}}$.
Since this is true between any consecutive times at which the orbit is located on $x = \alpha$,
it is also true for evolution of the attractor over all $t \in \mathbb{R}$.
Hence $\gamma_1 = \frac{T_{14} \tilde{\lambda}_{14}}{T_{14} + T_{23}}$.
By using (\ref{eq:T14}) and (\ref{eq:T23}) it is readily seen that this value limits to the value of $\gamma_1$ in (\ref{eq:gammaphizero}) as $\phi \to 0$.
The values of $\gamma_2$, $\gamma_3$ and $\gamma_4$ in (\ref{eq:gammaphizero}) follow in the same fashion from the above observations.
\end{proof}

\subsection{Chaotic dynamics induced by hysteresis}
\label{sub:chaos}

\begin{figure}[b!]
\begin{center}
\setlength{\unitlength}{1cm}
\begin{picture}(15.1,5)
\put(0,0){\includegraphics[height=4.7cm]{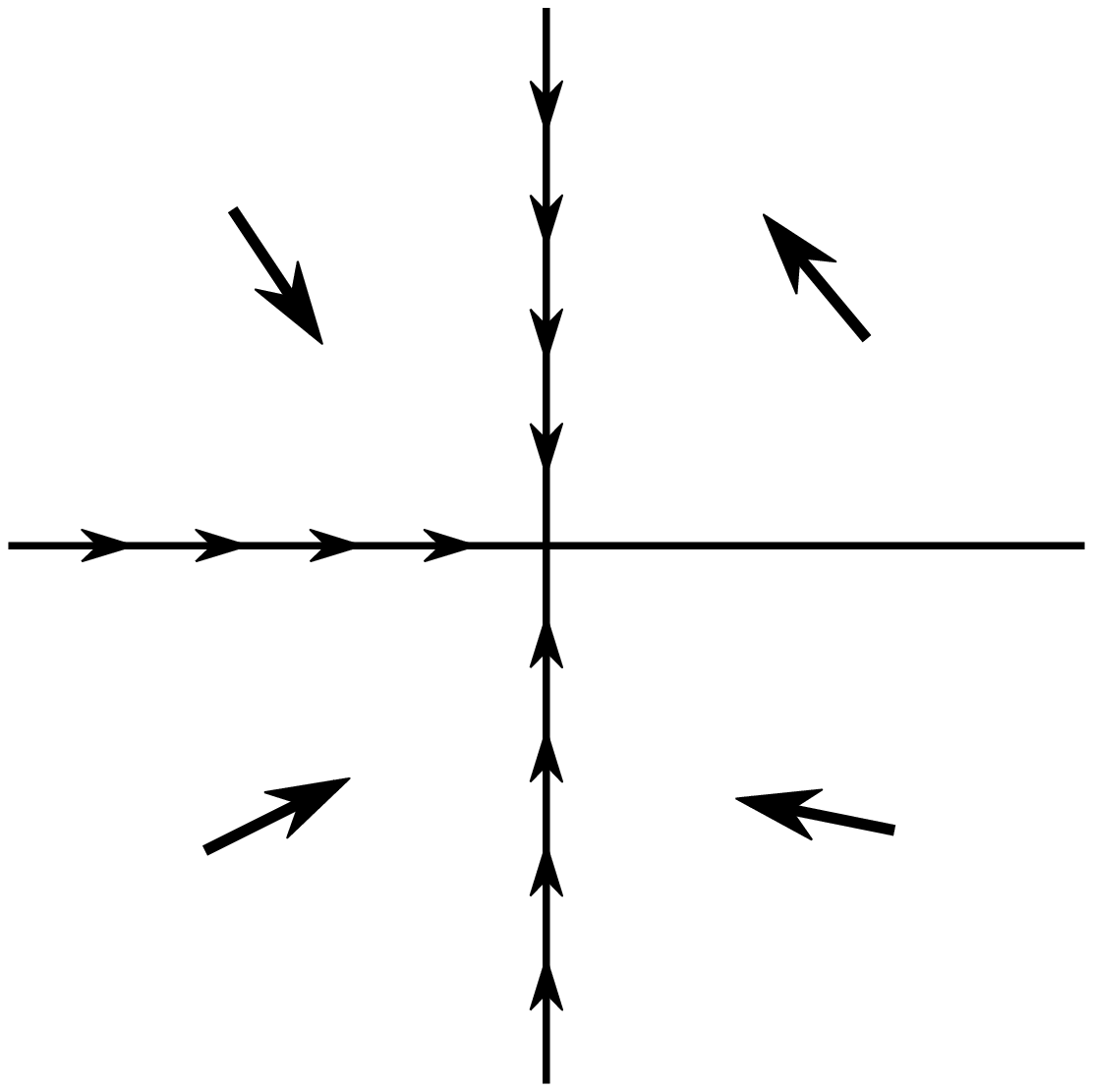}}
\put(5.2,0){\includegraphics[height=4.7cm]{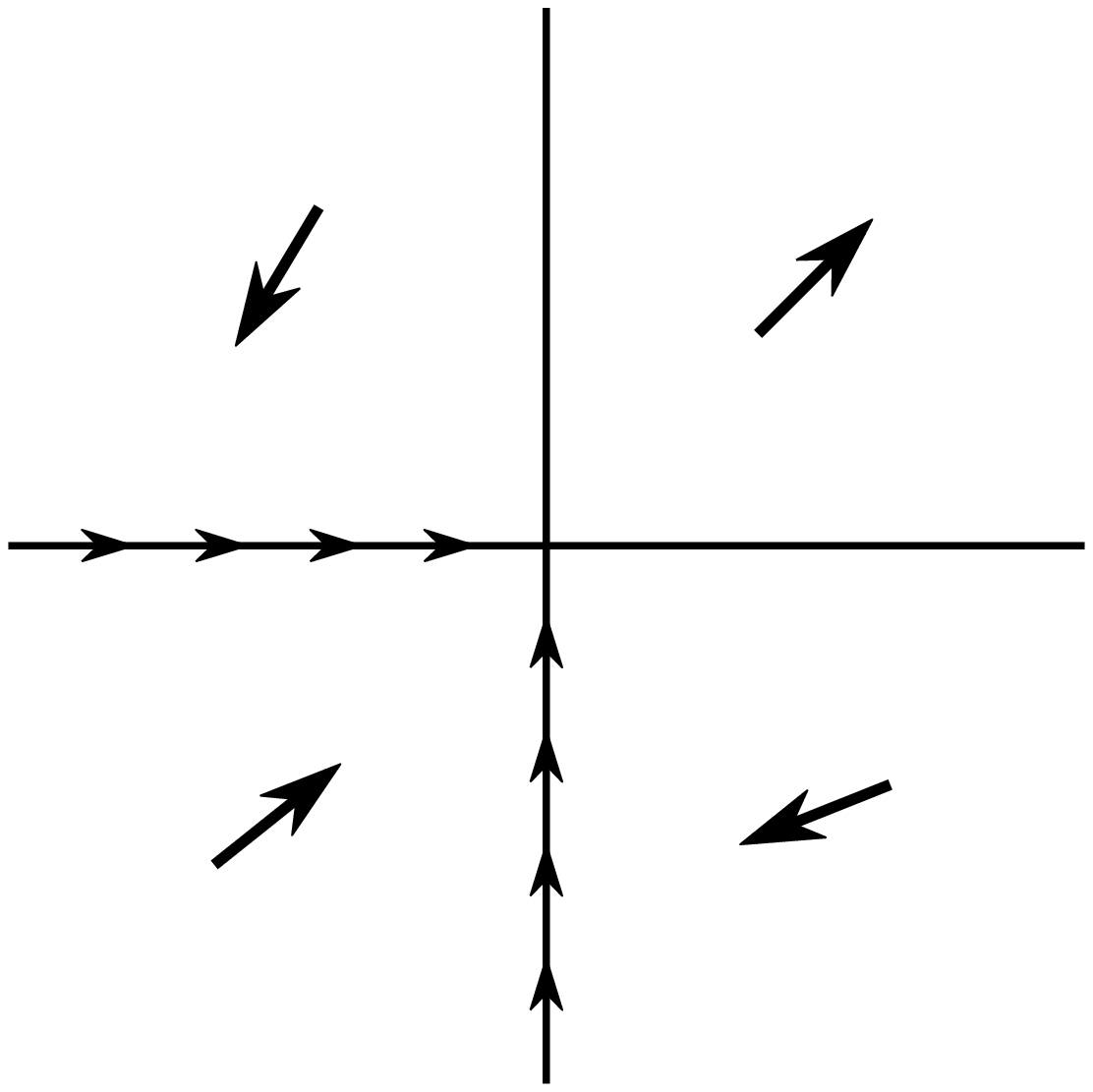}}
\put(10.4,0){\includegraphics[height=4.7cm]{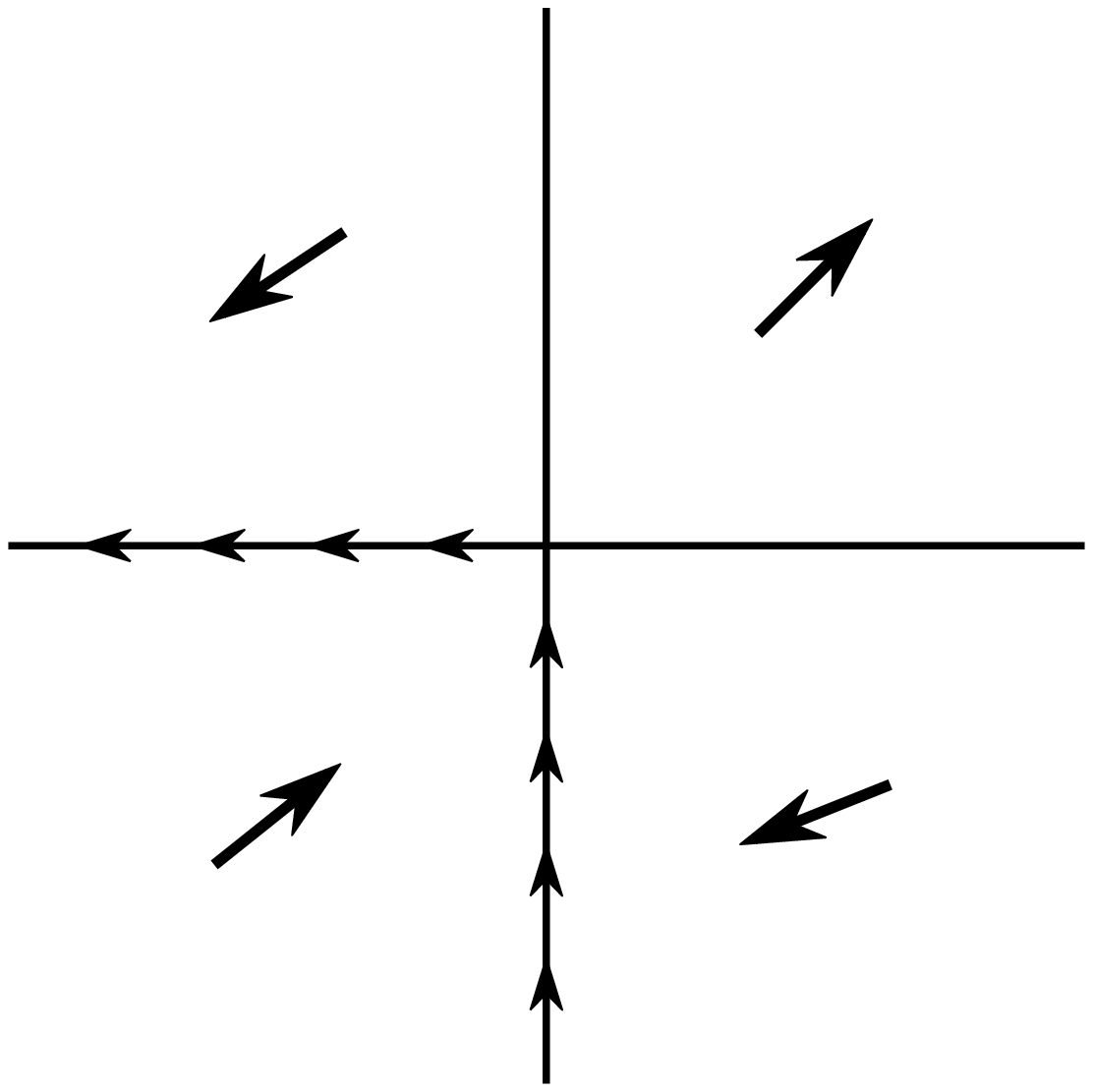}}
\put(.5,4.7){\large \sf \bfseries A}
\put(4.4,2.45){\small $x$}
\put(2.42,4.5){\small $y$}
\put(5.7,4.7){\large \sf \bfseries B}
\put(9.6,2.45){\small $x$}
\put(7.62,4.5){\small $y$}
\put(10.9,4.7){\large \sf \bfseries C}
\put(14.8,2.45){\small $x$}
\put(12.82,4.5){\small $y$}
\end{picture}
\caption{
Schematics of the two-dimensional piecewise-constant vector fields (\ref{eq:f2dPWC}) considered in
\sect{sub:chaos} (panel A), \sect{sub:stabilization} (panel B), and \sect{sub:exitSelection} (panel C).
Each quadrant has one arrow indicating the direction of flow.
The direction of Filippov's sliding solution is indicated on
the parts of the discontinuity surfaces ($x=0$ and $y=0$) that are attracting.
\label{fig:schemPWCVectorField}
} 
\end{center}
\end{figure}


We now illustrate hysteresic dynamics for piecewise-constant vector fields
in three cases for which one or more of the $\f fi$ is not directed inwards.
These are depicted by Fig.~\ref{fig:schemPWCVectorField}.

We first consider
\begin{equation}
\begin{aligned}
a_1 &= -1 \;, & a_2 &= 1 \;, & a_3 &= 1 \;, & a_4 &= -1 \;, \\
b_1 &= 1.2 \;, & b_2 &= -1.5 \;, & b_3 &= 0.5 \;, & b_4 &= 0.2 \;,
\end{aligned}
\label{eq:params36}
\end{equation}
see Fig.~\ref{fig:schemPWCVectorField}-A.
Here $\f f2$, $\f f3$ and $\f f4$ are directed inwards, but $\f f1$ is not.
Yet, for the unperturbed system (\ref{eq:f2dPWC}), the origin is a global attractor.
This is because the sliding dynamics on $x = 0$ with $y > 0$ approaches $(x,y) = (0,0)$
(specifically, the $y$-component of $\f f{12}$ is $\frac{a_2 b_1 - a_1 b_2}{a_2 - a_1} = -0.15$ which is negative-valued).

When the system is perturbed by hysteresis,
orbits repeatedly escape $\Omega$ (because $\f f1$ is not directed inwards)
but remain within some neighbourhood of the origin (because the origin is an attractor of (\ref{eq:f2dPWC})).
This is shown in Fig.~\ref{fig:circleMap36}-A.

Despite repeatedly escaping $\Omega$,
the map $q$ can be applied to this example without modification
and is shown in Fig.~\ref{fig:circleMap36}-B.
Unlike when each $\f fi$ is directed inwards,
here $q$ is discontinuous and neither one-to-one nor onto.

The third iterate of $q$ is also shown in Fig.~\ref{fig:circleMap36}-B.
For the given parameter values (\ref{eq:params36}), and more generally for an open set of parameter values about (\ref{eq:params36}),
there exists a trapping region within which the third iterate is given by a two-piece piecewise-linear function,
as shown in the inset.
This is a skew tent map with slopes $1$ and $\frac{(b_1/a_1)(b_3/a_3)}{(b_2/a_2)(b_4/a_4)}$.
With (\ref{eq:params36}) the latter slope is $-2$.
As described in \cite{MaMa93,NuYo95,DiBu08}, the dynamics is chaotic at these values.
Therefore, we have generated chaotic dynamics by incorporating hysteresis into the system.

\begin{figure}[h!]
\begin{center}
\setlength{\unitlength}{1cm}
\begin{picture}(15.17,5.8)
\put(0,0){\includegraphics[height=5.5cm]{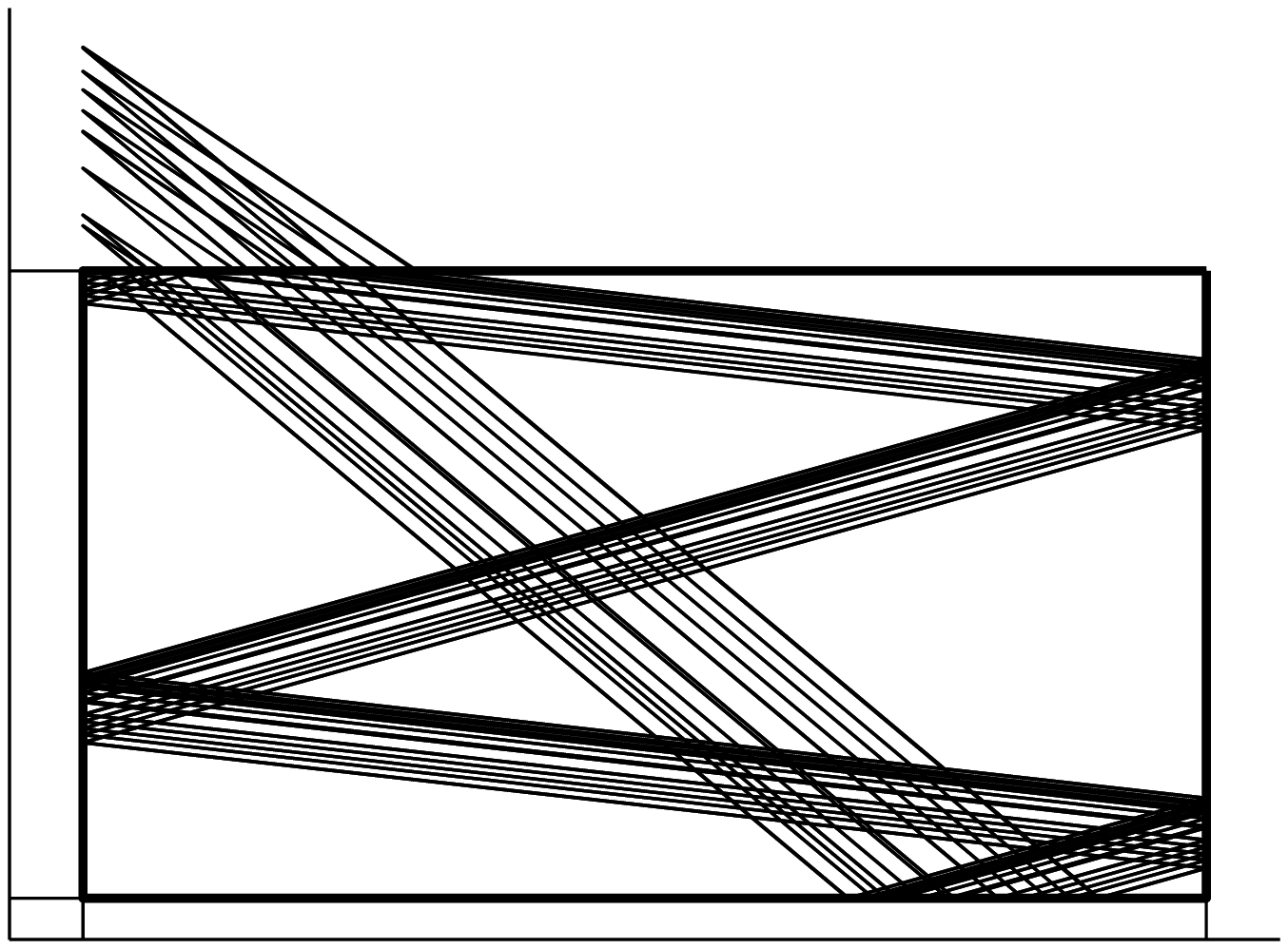}}
\put(7.83,0){\includegraphics[height=5.5cm]{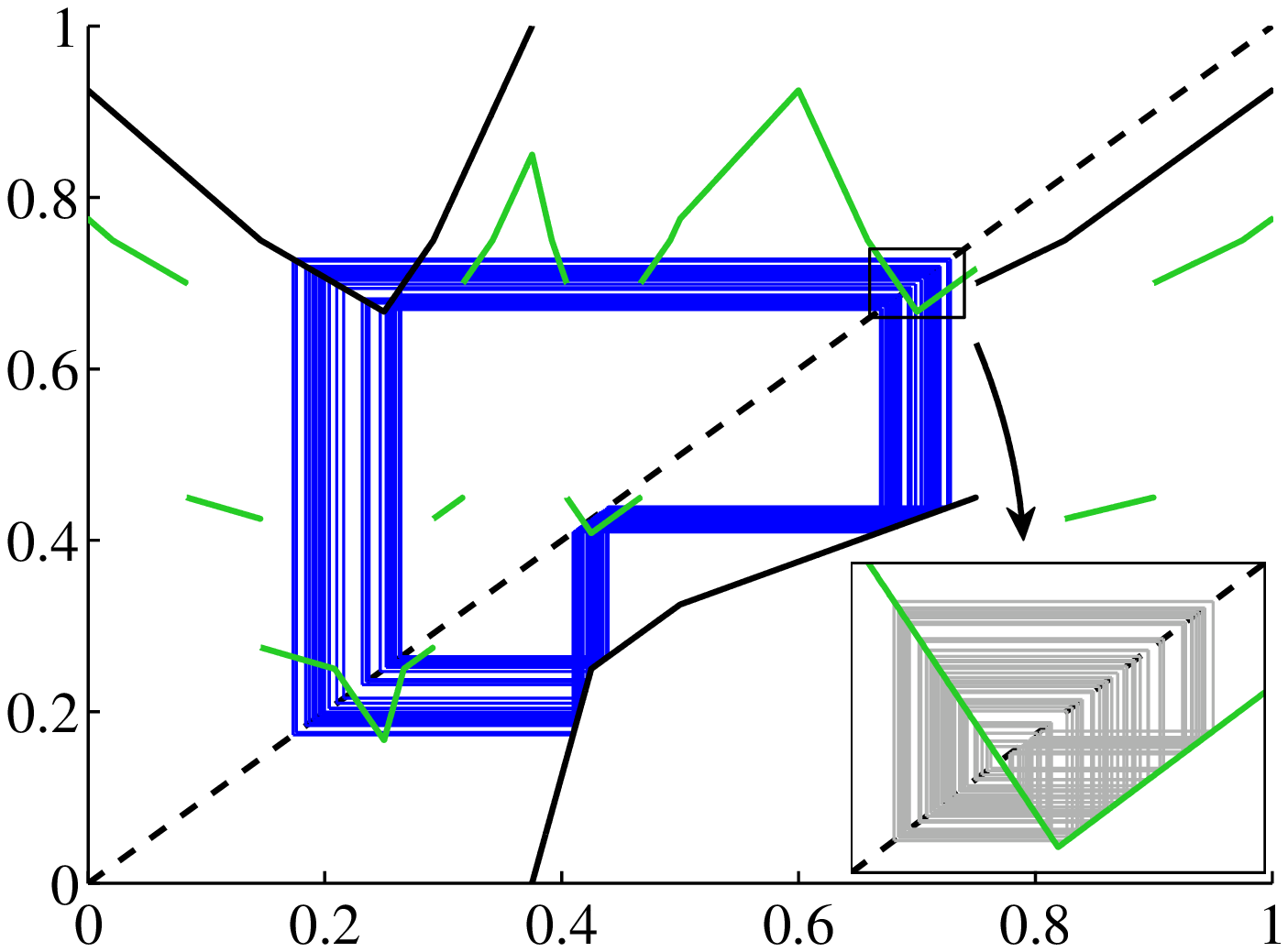}}
\put(.9,5.5){\large \sf \bfseries A}
\put(3.9,0){$x$}
\put(0,2.7){$y$}
\put(.66,.28){\small $-\alpha$}
\put(6.75,.28){\small $\alpha$}
\put(-.02,.66){\small $-\beta$}
\put(.28,3.96){\small $\beta$}
\put(8.73,5.5){\large \sf \bfseries B}
\put(11.73,0){$\eta$}
\put(11.01,5.1){\small $q$}
\put(12.64,4.8){\color{thirdIterate} \small $q^3$}
\end{picture}
\caption{
The chaotic attractor of (\ref{eq:f2dPWC}) with (\ref{eq:hysteresisRules}) and (\ref{eq:params36}).
Panel A shows part of an orbit in the $(x,y)$-plane;
panel B shows the circle map $q$ and its third iterate $q^3$.
\label{fig:circleMap36}
} 
\end{center}
\end{figure}

\subsection{Stabilization induced by hysteresis}
\label{sub:stabilization}


Here we consider (\ref{eq:f2dPWC}) with
\begin{equation}
\begin{aligned}
a_1 &= 1 \;, & a_2 &= -0.6 \;, & a_3 &= 1 \;, & a_4 &= -1 \;, \\
b_1 &= 1 \;, & b_2 &= -1 \;, & b_3 &= 0.8 \;, & b_4 &= -0.4 \;,
\end{aligned}
\label{eq:params50}
\end{equation}
see Fig.~\ref{fig:schemPWCVectorField}-B.
This a representative example, the exact values as given are not important and are only provided for clarity.
With these values $\f f1$ is directed outwards, hence $\Gamma$ is not an attracting sliding surface.
Yet the addition of hysteresis (\ref{eq:hysteresisRules}) causes orbits to become trapped near $\Gamma$.

To understand why this occurs, first notice that the values (\ref{eq:params50}) have been chosen
such that for the unperturbed system (\ref{eq:f2dPWC}),
on the negative $x$ and $y$-axes Filippov's sliding solution approaches the origin $(x,y) = (0,0)$.
Orbits thus approach the origin by either sliding along the negative $x$-axis,
sliding along the negative $y$-axis,
or regular motion in mode $3$. 

With the addition of hysteresis, such `approaching' dynamics involves
evolution with $x < \alpha$ and $y < \beta$ in modes $2$, $3$ and $4$.
But if an orbit is in mode $2$, since $\f f2$ points left and down,
the orbit cannot switch to mode $1$ by reaching $x = \alpha$, it
can only switch to mode $3$ by reaching $y = -\beta$.
Similarly, $\f f4$ points left and down
and so an orbit in mode $4$ can only switch to mode $3$.
If an orbit in mode $3$ reaches $y = \beta$ (with $x < \alpha$) it changes to mode $2$,
whilst if it reaches $y = \alpha$ (with $y < \beta$) it changes to mode $4$.
In the special case that it reaches $(x,y) = (\alpha,\beta)$ it changes to mode $1$ (and subsequently escapes).
Thus escape from a proximity to the origin requires passing through the point $(\alpha,\beta)$.
Hence, over any finite time interval, almost all orbits remain near the origin.
In this sense hysteresis stabilizes the unstable sliding surface $\Gamma$.


\subsection{Exit selection due to hysteresis}
\label{sub:exitSelection}

Lastly we consider
\begin{equation}
\begin{aligned}
a_1 &= 1 \;, & a_2 &= -1.5 \;, & a_3 &= 1 \;, & a_4 &= -1 \;, \\
b_1 &= 1 \;, & b_2 &= -1 \;, & b_3 &= 0.8 \;, & b_4 &= -0.4 \;,
\end{aligned}
\label{eq:params60}
\end{equation}
which is the same as (\ref{eq:params50}) except for the value of $a_2$, see Fig.~\ref{fig:schemPWCVectorField}-C.
Unlike the previous example, Filippov's sliding solution for (\ref{eq:f2dPWC}) on the negative
$x$-axis heads away from the origin.
Hence there are two routes by which orbits may travel away from the origin:
sliding motion along the negative $x$-axis
and regular motion in mode $1$.

In the presence of hysteresis, an orbit in mode $3$ eventually escapes proximity to the origin
by either switching about the negative $x$-axis or switching to mode $1$.
But as with the previous example, in order to switch to mode $1$ the orbit must pass through the point $(\alpha,\beta)$.
For this reason, almost all orbits escape along the negative $x$-axis instead of via mode $1$.

For the unperturbed system (\ref{eq:f2dPWC}), forward evolution from the origin is ambiguous.
But we can use the above observation to argue that, with the values (\ref{eq:params60}), forward evolution from the origin
should be given by Filippov's sliding solution along the negative $x$-axis.
That is, we dismiss the possibility of subsequent evolution in mode $1$.
Such exit selection has also been described for this situation by smoothing the vector field \cite{GuHa15,Je15e}.
It remains to formulate these ideas more generally.
This may have important consequences to the dynamics of gene networks, for example,
that involve switching at or near intersecting discontinuity surfaces \cite{EdGl14}.

\section{Other forms of regularisation: numerical experiments}
\label{sec:other}
\setcounter{equation}{0}


In this section we add time-delay, discretization and noise to the system (\ref{eq:f}).
As explained in \sect{sub:zeroPertLimit}, to determine the sliding solution
defined by taking the zero-perturbation limit in each case, we take the vector field $\v f$ to be piecewise-constant,
set the perturbation size to $\eps = 1$, and study attractors of the two-dimensional system (\ref{eq:f2dPWC}).
As shown in Fig.~\ref{fig:allBifDiag_z21}, the resulting sliding solution is jittery for time-delay and discretization, but not noise.
The three perturbations are more mathematically involved than hysteresis
and establishing rigorous results is beyond the scope of this paper.

\subsection{Time-delay}
\label{sub:delay}

Here we consider
\begin{equation}
\dot{\v x}(t) = \v f(\v x(t-\eps)) \;,
\label{eq:fTimeDelay}
\end{equation}
where $\eps > 0$ represents a constant time-delay.
This scenario was also considered briefly in \cite{AsIz89}.
One could more generally implement different size delays for the two switching conditions, but a single delay is sufficient here.

\begin{figure}[h!]
\begin{center}
\setlength{\unitlength}{1cm}
\begin{picture}(15.17,5.8)
\put(0,0){\includegraphics[height=5.5cm]{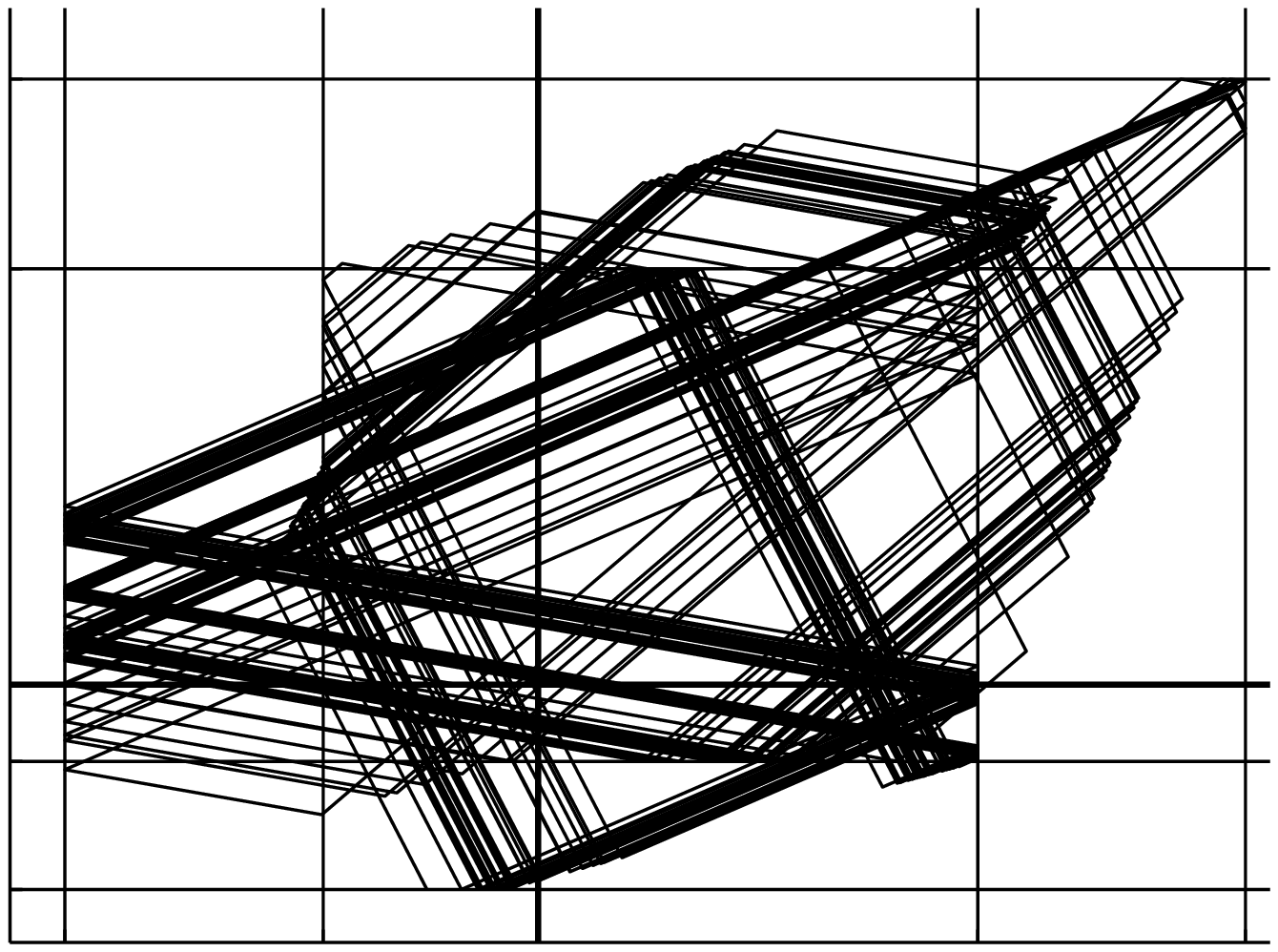}}
\put(7.83,0){\includegraphics[height=5.5cm]{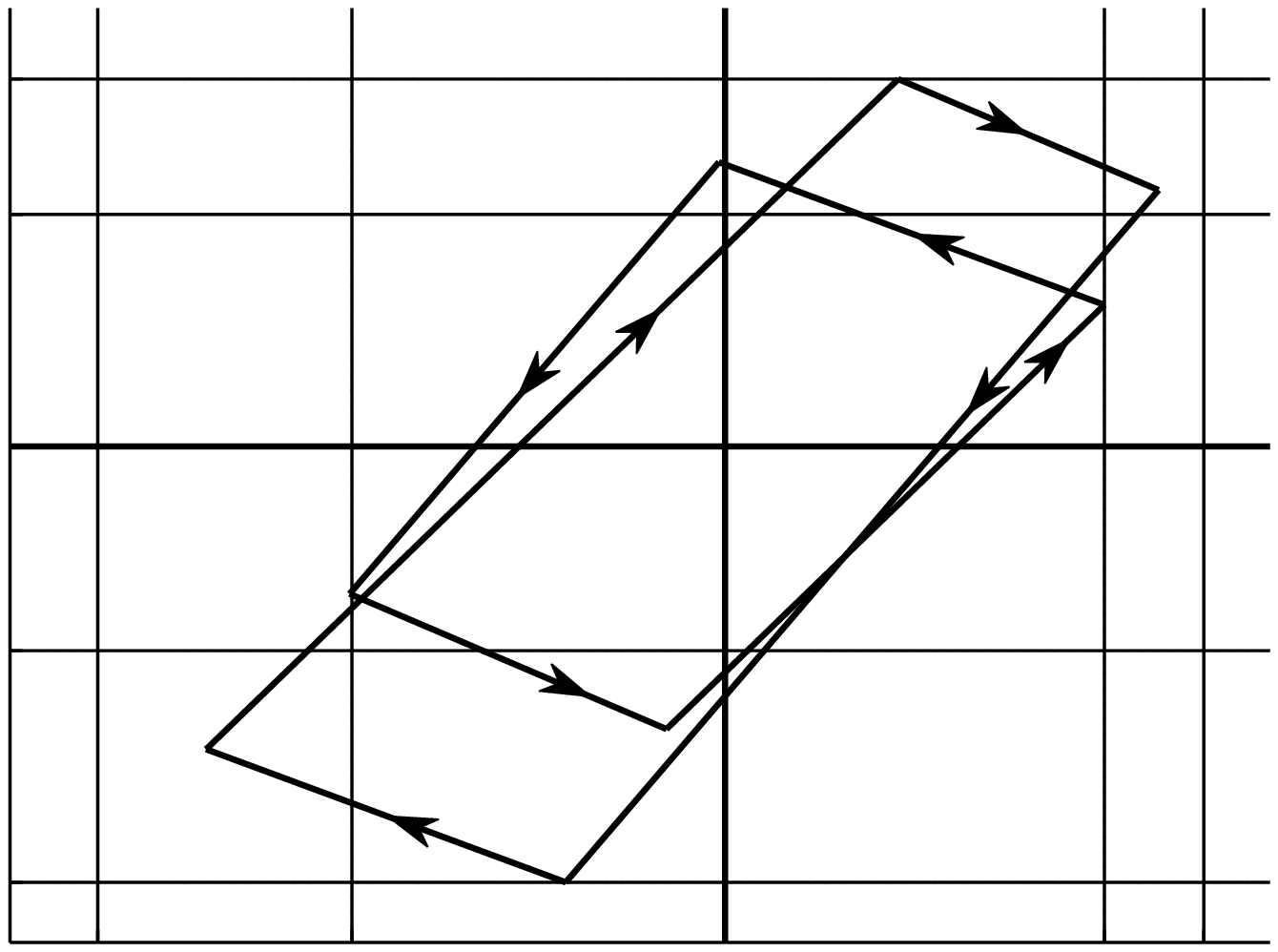}}
\put(.78,5.5){\large \sf \bfseries A}
\put(4,0){$x$}
\put(0,3){$y$}
\put(.82,.32){\footnotesize $a_1$}
\put(2.18,.32){\footnotesize $a_4$}
\put(3.34,.26){\footnotesize $0$}
\put(5.62,.32){\footnotesize $a_2$}
\put(7.03,.32){\footnotesize $a_3$}
\put(.31,.78){\footnotesize $b_1$}
\put(.31,1.43){\footnotesize $b_2$}
\put(.42,1.81){\footnotesize $0$}
\put(.31,4.04){\footnotesize $b_4$}
\put(.31,5.03){\footnotesize $b_3$}
\put(8.61,5.5){\large \sf \bfseries B}
\put(11.83,0){$x$}
\put(7.83,3){$y$}
\put(8.82,.32){\footnotesize $a_4$}
\put(10.16,.32){\footnotesize $a_1$}
\put(12.15,.26){\footnotesize $0$}
\put(14.13,.32){\footnotesize $a_3$}
\put(14.67,.32){\footnotesize $a_2$}
\put(8.14,.81){\footnotesize $b_1$}
\put(8.14,2.02){\footnotesize $b_2$}
\put(8.25,3.05){\footnotesize $0$}
\put(8.14,4.3){\footnotesize $b_4$}
\put(8.14,5.01){\footnotesize $b_3$}
\put(11.01,3.53){\scriptsize $1$}
\put(11.32,1.62){\scriptsize $2$}
\put(14.04,3.45){\scriptsize $3$}
\put(13.11,4.02){\scriptsize $4$}
\put(13.35,3.41){\scriptsize $1$}
\put(13.6,4.59){\scriptsize $2$}
\put(11.9,3.64){\scriptsize $3$}
\put(10.46,1.28){\scriptsize $4$}
\end{picture}
\caption{
Attractors of the two-dimensional piecewise-constant system (\ref{eq:f2dPWC}) with $\eps = 1$ time-delay
for the random example (\ref{eq:fghLinearInt})-(\ref{eq:abc1})
with $z = 0$ in panel A and $z = 0.9$ in panel B.
In panel B there are two attracting periodic orbits.
The different parts of these orbits are labelled by mode.
\label{fig:timeDelayPhasePortrait21}
} 
\end{center}
\end{figure}

With each $\f fi$ directed inwards, orbits of (\ref{eq:f2dPWC}) with (\ref{eq:fTimeDelay}) become trapped in a neighbourhood of the origin.
This is shown for two examples in Fig.~\ref{fig:timeDelayPhasePortrait21}.
The orbits change mode whenever $x(t-\eps) = 0$ or $y(t-\eps) = 0$ (with $\eps = 1$).
This occurs on the eight lines $x = \ff ai$ and $y = \ff bi$, for $i = 1,\ldots,4$,
as well as at other points in cases for which both $x = 0$ and $y = 0$ are crossed in a time less than $\eps$.

We obtained the zero-time-delay sliding solution shown in Fig.~\ref{fig:allBifDiag_z21}
by computing the forward orbit of the origin in mode $1$, and removing transient dynamics,
in order to identify an attractor of the system.
As with hysteresis, we observed that typically the attractor is periodic.
With $0.74 < z < 0.83$, approximately, the attractor does not involve mode $4$.
In this interval the sliding solution therefore has $\gamma_4 = 0$ and lies on the boundary of the convex hull.

Unlike with hysteresis, the time-delayed system can have multiple attractors
and hence multiple sliding solutions in the zero-perturbation limit.
With $z = 0.9$, for example, we have identified two distinct attractors, Fig.~\ref{fig:timeDelayPhasePortrait21}-B.
The presence of multiple attractors could allow for interesting (e.g.~periodic) sliding dynamics on $\Gamma$.
This is because the attractors undergo bifurcations as the value of $z$ is varied
and therefore sliding orbits could switch between different sliding solutions as these bifurcation values are reached.
Such complexities are left for future work.

\subsection{Numerical discretisation}
\label{sub:discrete}

Arguably the simplest and most direct manner by which to
numerically compute orbits to systems of discontinuous differential equations
is to use a numerical method with no special attention paid to the discontinuity surfaces.
Regardless of theoretical results on the nature of solutions to discontinuous differential equations,
it is highly important to understand how such numerical solutions behave
as these are the solutions that applied scientists would most naturally obtain.

Here we consider forward Euler with step-size $\eps$
\begin{equation}
\v x_{i+1} = \v x_i + \eps \v f(\v x_i) \;.
\label{eq:forwardEuler}
\end{equation}
Note that for discontinuous differential equations some methods, such as backward Euler, are ill-posed.

About a single attracting discontinuity surface,
(\ref{eq:forwardEuler}) generates a solution that rapidly switches back-and-forth across the surface.
As $\eps \to 0$ this solution converges to Filippov's solution, see for instance \cite{DoLe92}.

Fig.~\ref{fig:allBifDiag_z21} shows our result using (\ref{eq:forwardEuler})
to generate a sliding solution for the system (\ref{eq:f}) with (\ref{eq:fghLinearInt})-(\ref{eq:abc1}).
To achieve this, for each value of $z \in [0,1]$ (we used values spaced by $\Delta z = 0.001$),
we constructed the two-dimensional piecewise-constant system (\ref{eq:f2dPWC}).
We then computed the forward orbit of the origin in mode $1$
for (\ref{eq:f2dPWC}) using (\ref{eq:forwardEuler}) (with $\eps = 1$ due to the scaling invariance)
for $10^5$ time-steps.
We then used this orbit, with transient dynamics removed, to evaluate (\ref{eq:hSlide})
where each $\gamma_i$ is the fraction of points of the orbit that lie in ${\cal Q}_i$,
and this way produced Fig.~\ref{fig:allBifDiag_z21}.
We notice that $h_{\rm slide}$ is a highly erratic function of $z$
thus producing a jittery sliding solution.

The system (\ref{eq:f2dPWC}) with (\ref{eq:forwardEuler}) is a two-dimensional piecewise-smooth map.
Moreover, the map is discontinous and each of the four smooth pieces of the map is a translation.
With each $\f fi$ directed inwards, all forward orbits become trapped in a neighbourhood of the origin.
The forwards orbits are typically aperiodic and evenly fill a dense subset of some patterned region, call it $\Psi$.
Fig.~\ref{fig:timeDiscPhasePortrait21} shows $\Psi$ for two examples.

Our numerical investigations reveal that $\Psi$ is often, but not always, unique.
In this case $\Psi$ appears to be the closure of the $\omega$-limit set of every initial point in the $(x,y)$-plane.
The $\gamma_i$ 
are then given by the fraction of $\Psi$ contained in each quadrant of the $(x,y)$-plane.
The region $\Psi$ can change smoothly with $z$,
but also undergo fundamental changes due to interactions with $x=0$ and $y=0$.
A further study of the properties of $\Psi$ remains for future work.

\begin{figure}[h!]
\begin{center}
\setlength{\unitlength}{1cm}
\begin{picture}(15.17,5.8)
\put(0,0){\includegraphics[height=5.5cm]{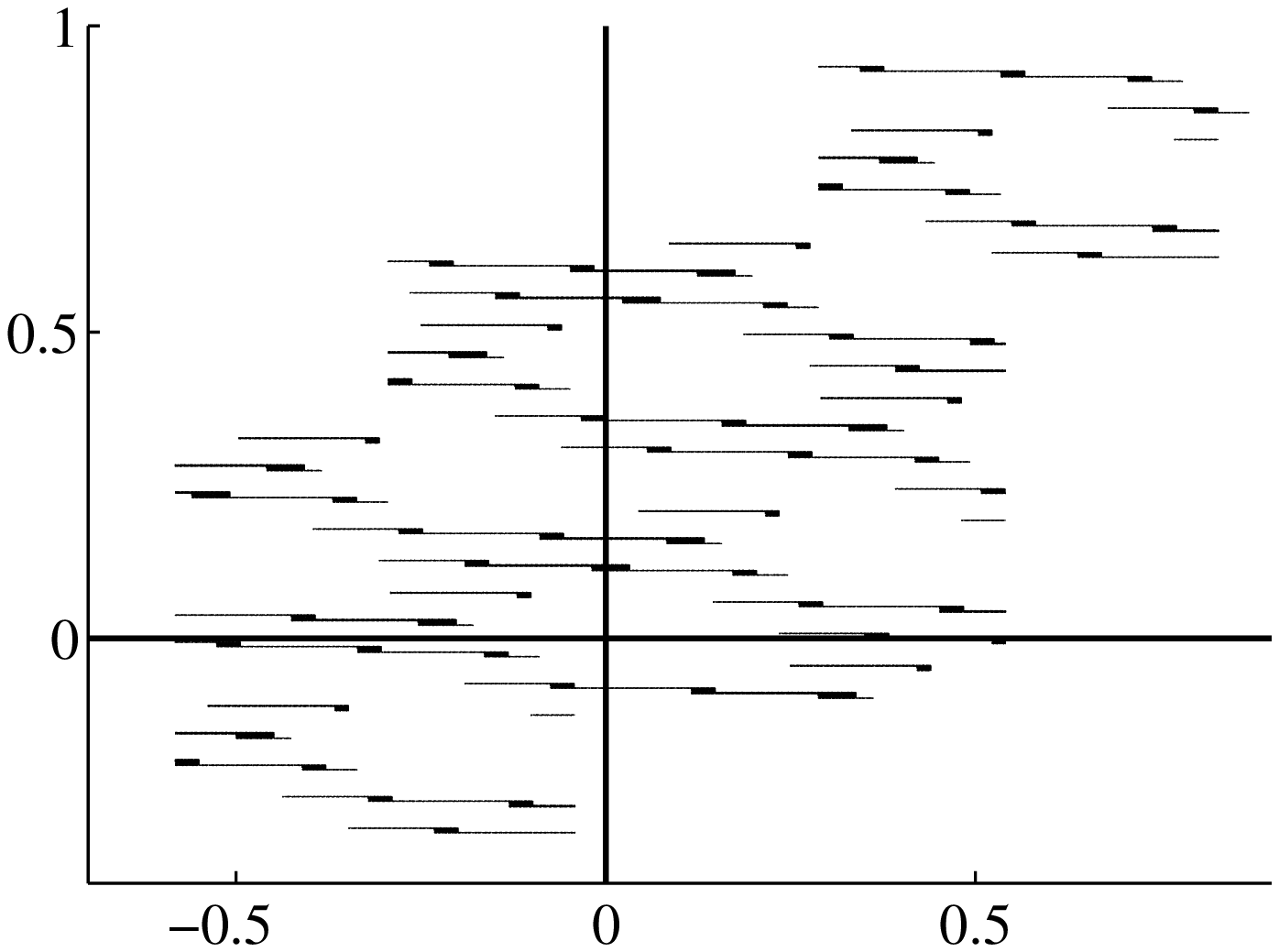}}
\put(7.83,0){\includegraphics[height=5.5cm]{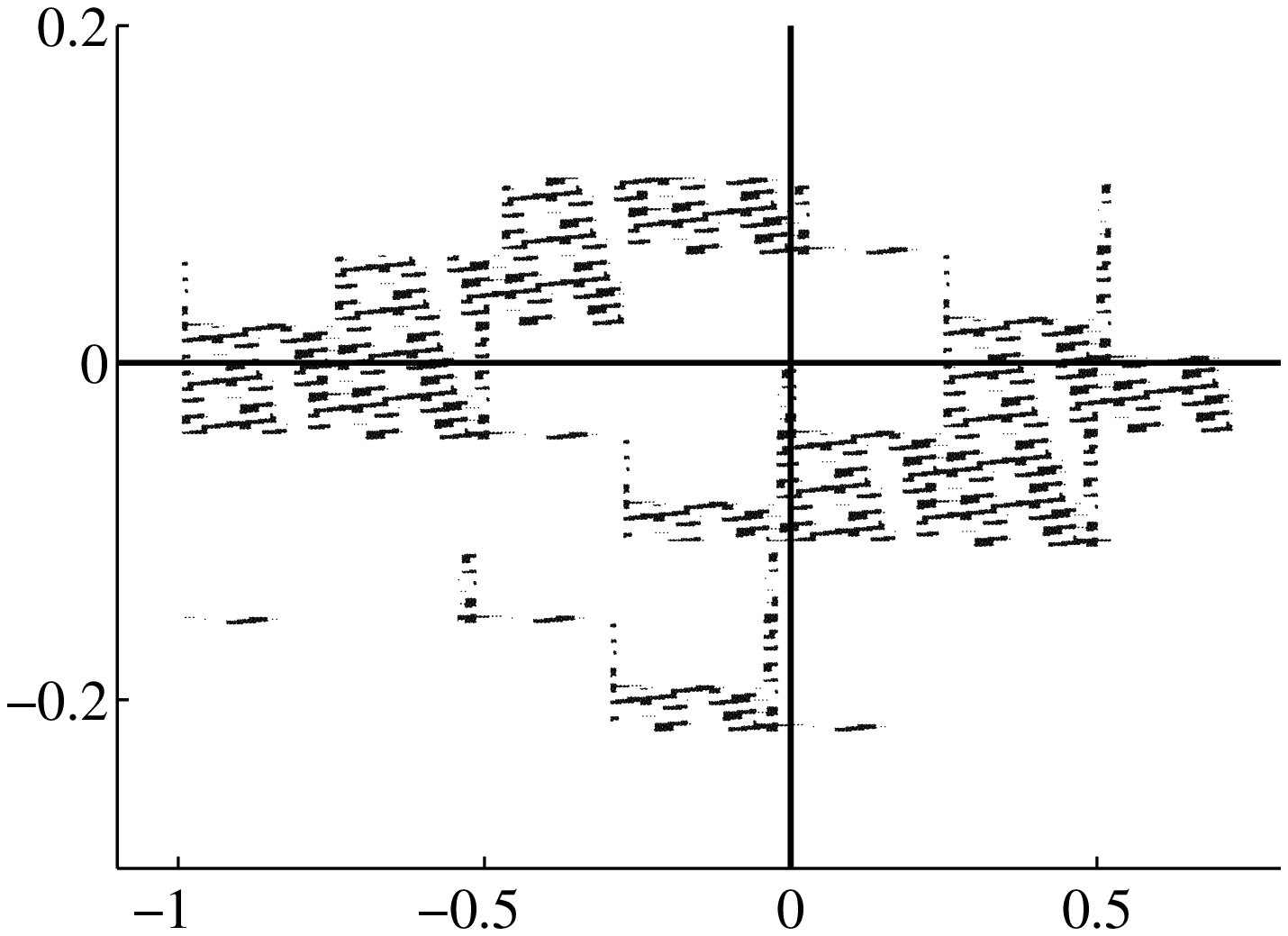}}
\put(.9,5.5){\large \sf \bfseries A}
\put(4,0){$x$}
\put(0,3){$y$}
\put(8.73,5.5){\large \sf \bfseries B}
\put(11.83,0){$x$}
\put(7.83,3){$y$}
\end{picture}
\caption{
Attractors $\Psi$ of the two-dimensional piecewise-constant system (\ref{eq:f2dPWC})
discretised by forward Euler (\ref{eq:forwardEuler}) with $\eps = 1$
for the random example (\ref{eq:fghLinearInt})-(\ref{eq:abc1}).
In panel A, $z = 0$; in panel B, $z = 1$.
\label{fig:timeDiscPhasePortrait21}
} 
\end{center}
\end{figure}

\subsection{Noise}
\label{sub:noise}

Here we add noise to define a sliding solution on $\Gamma$.
This was done for some simple examples exhibiting symmetry in \cite{AsIz89}.
The idea of using noise to resolve an ambiguity in forward evolution has been used previously for
non-Lipschitz points of continuous vector fields \cite{BaBa82,FlLa08,BoSu10},
as well as two-folds of discontinuous vector fields \cite{Si14c,SiJe16}.

We consider
\begin{equation}
d\v x(t) = \v f(\v x(t)) \,dt + \eps D \,d\v W(t) \;,
\label{eq:sde}
\end{equation}
where $\v W(t)$ is a standard two-dimensional vector Brownian motion,
and $D$ is a non-singular matrix that allows for different noise magnitudes in different directions..
As described above, in order to determine the value of $h_{\rm slide}$ in the $\eps \to 0$ limit,
we work with (\ref{eq:f2dPWC}) and can set $\eps = 1$.

Despite the discontinuities in $\v f$,
the system (\ref{eq:f2dPWC}) with (\ref{eq:sde}) has a unique stochastic solution \cite{StVa69,KrRo05}.
To produce the zero-noise solution shown in Fig.~\ref{fig:allBifDiag_z21},
for each value of $z$ we computed a sample solution to (\ref{eq:f2dPWC}) with (\ref{eq:sde}) and $D = I$
using the Euler-Maruyama method,
let $\gamma_i$ be the fraction of time that the orbit spent in ${\cal Q}_i$, for each $i$,
and evaluated (\ref{eq:hSlide}).
This sliding solution appears to be a smooth function of $z$,
thus not displaying jitter.

To further understand the origin of this sliding solution,
let $p_{\rm trans}(\v x,t;\v x_0)$ denote the transitional
probability density function for (\ref{eq:f2dPWC}) with (\ref{eq:sde}).
That is, given $\v x(0) = \v x_0$,
for any measurable subset $E \subset \mathbb{R}^2$ and any $t > 0$,
the probability that $\v x(t) \in E$ is $\int_E p_{\rm trans}(\v x,t;\v x_0) \,d \v x$.
If each $\f fi$ is directed inwards, then $p_{\rm trans}$ converges to
a steady-state density $p(\v x)$ as $t \to \infty$, see Fig.~\ref{fig:noiseSteadyState21}.
Since (\ref{eq:f2dPWC}) with (\ref{eq:sde}) is ergodic \cite{Sk89,Kh10,CaBa15},
the fraction of time $\gamma_i$ spent in ${\cal Q}_i$
is equal to the spatial fraction of $p$ over ${\cal Q}_i$.
That is,
\begin{equation}
\gamma_i = \int_{{\cal Q}_i} p(\v x) \,d \v x \;, \qquad i = 1,\ldots,4 \;.
\end{equation}

Finally, we formulate a boundary value problem for $p$.
It is a steady-state solution to the Fokker-Planck equation
of (\ref{eq:f2dPWC}) with (\ref{eq:sde}), that is
\begin{equation}
-\nabla \cdot (\v f p) + \frac{\eps^2}{2} \nabla^{\sf T} D D^{\sf T} \nabla p = 0 \;, \qquad x \ne 0 \;, y \ne 0 \;.
\end{equation}
Along $x=0$ and $y=0$ the density $p$ is continuous
and the `flow of probability' across these boundaries is the same on each side.
That is, the left and right limiting values of the probability current
\begin{equation}
J = \v n^{\sf T} \left( \v f p - \frac{\eps^2}{2} D D^{\sf T} \nabla p \right) \;,
\end{equation}
are equal, where $\v n$ is a unit normal vector to the boundary.
This condition specifies the jump in the derivative of $p$ at $x = 0$ and $y = 0$.
Also $p \to 0$ as $|\v x| \to \infty$.

Despite having a piecewise-constant drift vector $\v f$ and a constant diffusion matrix $D$,
we have been unable to obtain an analytical solution to this boundary value problem.
This problem remains for future work.

\begin{figure}[h!]
\begin{center}
\setlength{\unitlength}{1cm}
\begin{picture}(15.17,5.8)
\put(0,0){\includegraphics[height=5.5cm]{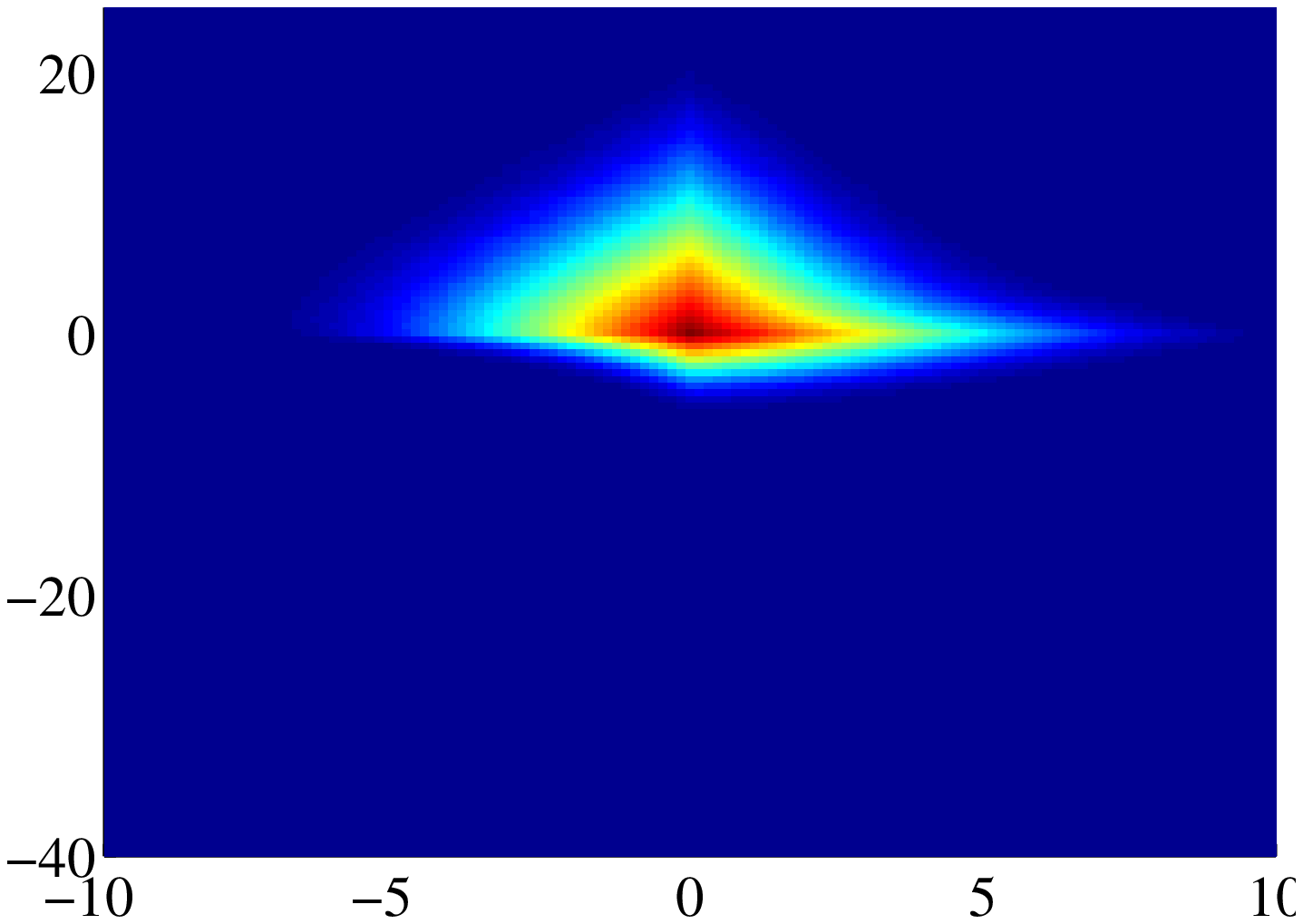}}
\put(7.83,0){\includegraphics[height=5.5cm]{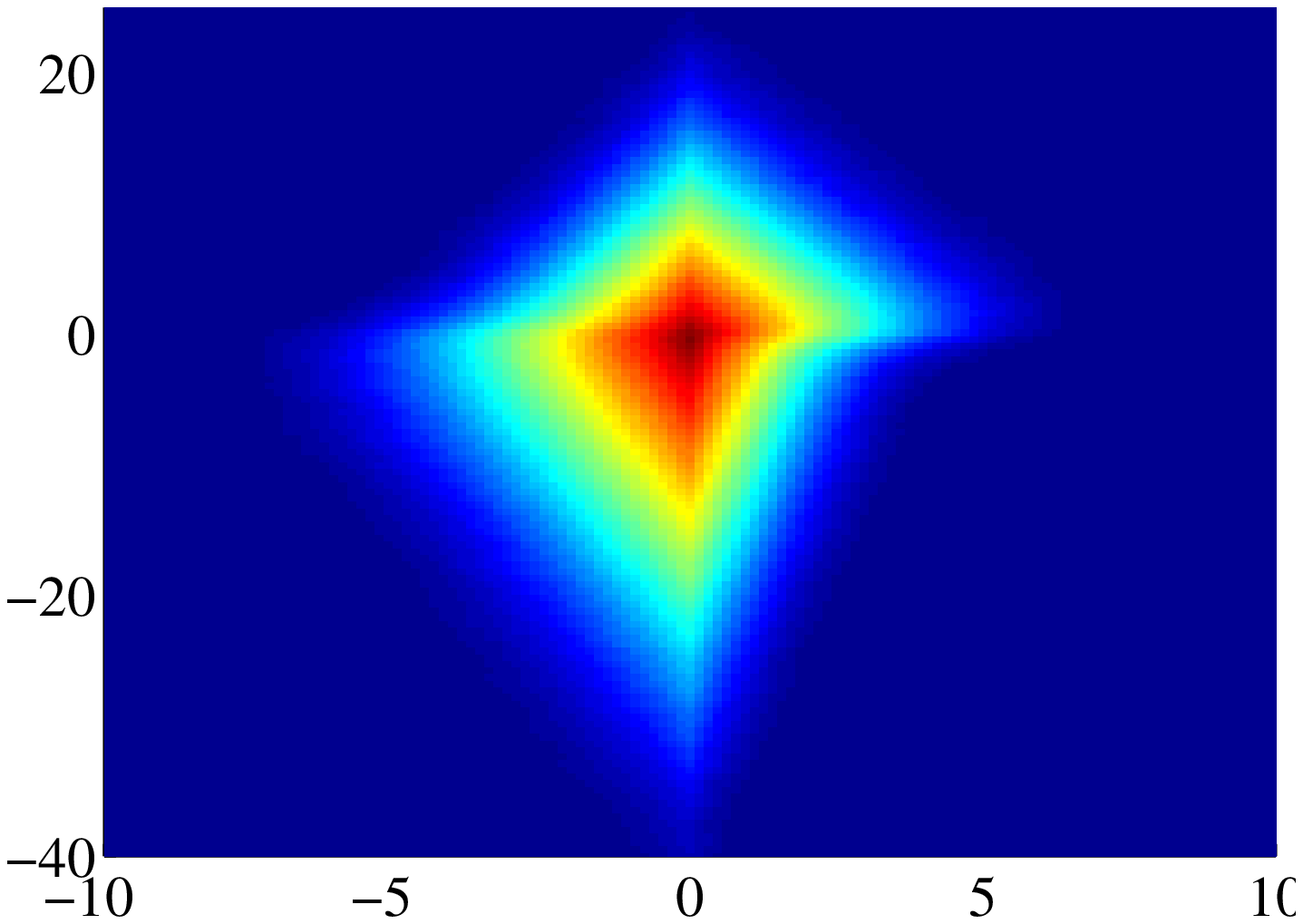}}
\put(.9,5.5){\large \sf \bfseries A}
\put(4,0){$x$}
\put(0,3){$y$}
\put(8.73,5.5){\large \sf \bfseries B}
\put(11.83,0){$x$}
\put(7.83,3){$y$}
\end{picture}
\caption{
Steady state probability density functions for the
two-dimensional piecewise-constant system (\ref{eq:f2dPWC})
perturbed by noise (\ref{eq:sde}) using $D = I$ and $\eps = 1$
for the random example (\ref{eq:fghLinearInt})-(\ref{eq:abc1}).
In panel A, $z = 0$; in panel B, $z = 1$.
The value of the density is indicated by shading
(dark red --- maximum value, dark blue --- zero) [color online].
\label{fig:noiseSteadyState21}
} 
\end{center}
\end{figure}

\section{An applied example - interfacing energy sources with an inductor}
\label{sec:eg}
\setcounter{equation}{0}

Power electronics use electrical switches to regulate electrical energy conversion. 
The physical switches used in virtually all modern implementations of power electronics are solid state semiconductor devices. Rapid advancements in the performance of these has lead to significant improvements in performance and efficiency. As a result, electrical power is becoming almost the universal means to transfer power in engineering applications. Power electronics using discontinuous switching \cite{Utkin1993,Lai2008} perform the electrical energy conversion that is necessary to interface machinery with electrical power transfer systems.

In a simple example we investigate how an inductor, two capacitors, and electrical switches can be used to interface four electrical energy sources in a theoretically lossless manner. This device is a basic example of a two-switch topology (e.g. \cite{Sira-Ramirez1987}), and because of the independence of the switches, is found to exhibit jitter. 

We assume that there are two voltage sources, $E_{1}$ and $E_{2}$, and two current sources, $I_{1}$ and $I_{2}$. There are four possible configurations of the system that are selected according to two independent switching surfaces. The equations describing the system are 
\begin{align}\label{circuit}
	L \frac{di_{L}}{dt} &= - (1-u_{1}) (1-u_{2}) v_{C_{1}} - (1-u_{1}) u_{2} E_{1} + u_{1} (1-u_{2}) E_{2} + u_{1} u_{2} v_{C_{2}} \\
	C_{1} \frac{dv_{C_{1}}}{dt} &= (1-u_{1}) (1-u_{2}) i_{L} - (1-u_{2}) I_{1} + u_{2} I_{2} \\
	C_{2} \frac{dv_{C_{2}}}{dt} &= - u_{1} u_{2} i_{L} + u_{2} I_{1} - (1-u_{2}) I_{2}
\end{align}
where $u_1$ and $u_2$ take the values $0$ or $1$. If these are chosen as
\begin{align}
	u_{1} &= \step(i_{L}^{*}-i_{L}) & u_{2} &= \step(v_{C_1}^{*}- v_{C_1} ),
\end{align}
in terms of the Heaviside step function with $\step(h)=0$ for $h<0$ and $\step(h)=1$ for $h>0$, 
then the values of $i_{L}$ and $v_{C_{1}}$ are controllable provided $i_{L},i_{L}^{*} < I_{1}$. 
An example of such a switched mode circuit with single-pole double-through switches is shown in Fig. \ref{Fig:InductiveTransfer}. 

Letting $x=(i_{L}-i_{L}^*)/A$, $y=(v_{C_1}-v_{C_1}^{*})/V$, $z=v_{C_2}/V$, where $A$ and $V$ denote the units of current and voltage, and taking physically reasonable parameter values $L=1.5mH$, $C_1=1.6mF$, $C_2=1.2mF$, $i_l^*=3A$, $v_C^*=10V$, $E_1=12V$, $E_2=10V$, $I_1=6A$, $I_2=4A$, the four modes become 
\begin{equation}
(\dot x,\dot y,\dot z) = 10^{3} \times \begin{cases}
(-2(y+10)/3,\;5(x-3)/8,\;-10/3) \;, & (x,y) \in {\cal Q}_1 \;, \\
(20/3,\;-15/4,\;-10/3) \;, & (x,y) \in {\cal Q}_2 \;, \\
(2z/3,\;5/2,\;5(3-x)/6) \;, & (x,y) \in {\cal Q}_3 \;, \\
(-8,\;5/2,\;5) \;, & (x,y) \in {\cal Q}_4 \;.
\end{cases}
\end{equation}

We take hysteresis bounds $\alpha=\beta=\eps/\sqrt2$ (so $\phi=1/2$ in the notation of (\ref{eq:alphabeta})). The simulations below are obtained approximating $x\approx y\approx0$, which is valid for sufficiently small $\eps$. In these simulations we use $\eps=\sqrt2/10$ or $\eps=\sqrt2/10^3$ (so that $\alpha=\beta=0.1$ or $0.001$). 

%

\begin{figure}[h!]
	\begin{center}
		\includegraphics[width=0.65\textwidth]{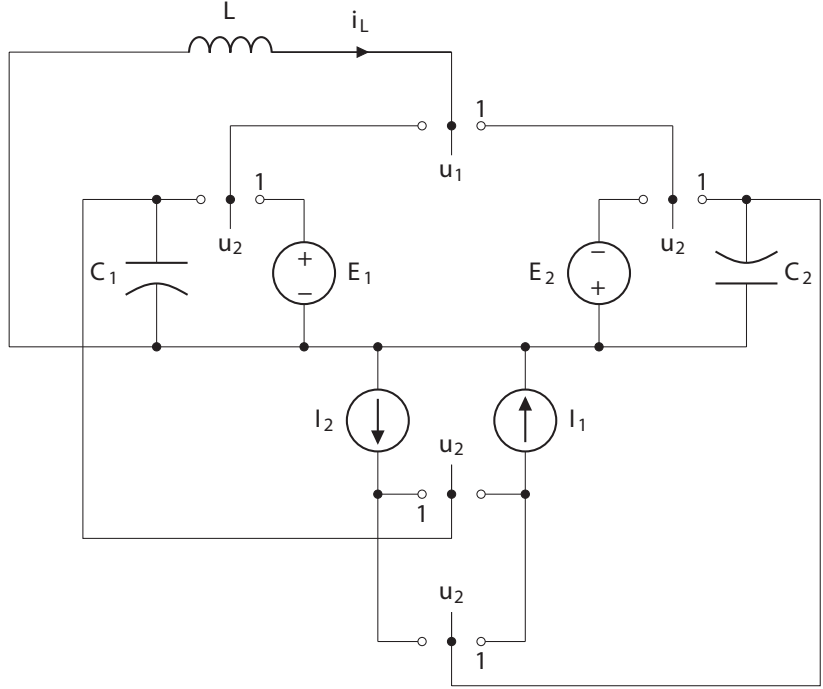}
      	\caption{Power transfer between sources using a switch network and passive elements. The variable controlling each switch $(u_{1}, u_{2})$ is shown above the switch itself; the position that the switch takes for input $1$ is also marked. For input $0$ the switch takes the complementary position.}
      	\label{Fig:InductiveTransfer}
	\end{center}
\end{figure}	

\begin{figure}[h!]
	\begin{center}
		\includegraphics[width=0.6\textwidth]{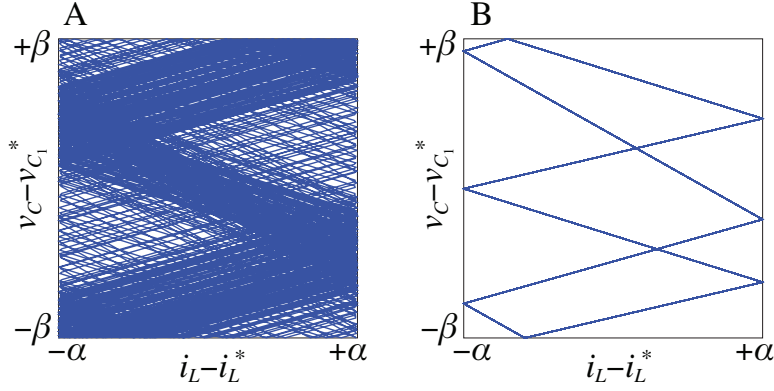}
      	\caption{Attractors inside the $\varepsilon$-hysteresis chatterbox at $z=v_{C_2}=13$ (panel A) and $z=v_{C_2}=16$ (panel B), with $\alpha=\beta=0.1$. 
}
      	\label{Fig:circjitter}
	\end{center}
\end{figure}	

Figure \ref{Fig:circjitter} shows two examples of the attractor between the hysteresis boundaries at fixed values of $z$. 
Figure \ref{Fig:circsim} shows a simulation of the system's orbit found by solving the full system of equations (\ref{circuit}). The effect of jitter is observable as a marked change in the gradient, particularly around $t\approx0.032s$. Two simulations for different hysteresis widths $\eps$ are shown, with similar results. The gradient, corresponding to the speed $\dot z=h_{\rm slide}$, can be calculated numerically by taking the gradient of this graph, shown in panel B by the green dotted and red dashed curves for the two curves in panel A. This is compared in panel B to the theoretical sliding speed, found by iterating the hysteretic map inside the chatterbox at fixed $z$ (thin blue curve), which corresponds to taking an ideal limit $\eps=0$, and which is almost indistinguishable from the $\eps=\sqrt2/10^3$ simulation. Also shown in Panel B are the canopy (black curve) and the hull (unshaded region). 
\begin{figure}[h!]
	\begin{center}
		\includegraphics[width=0.95\textwidth]{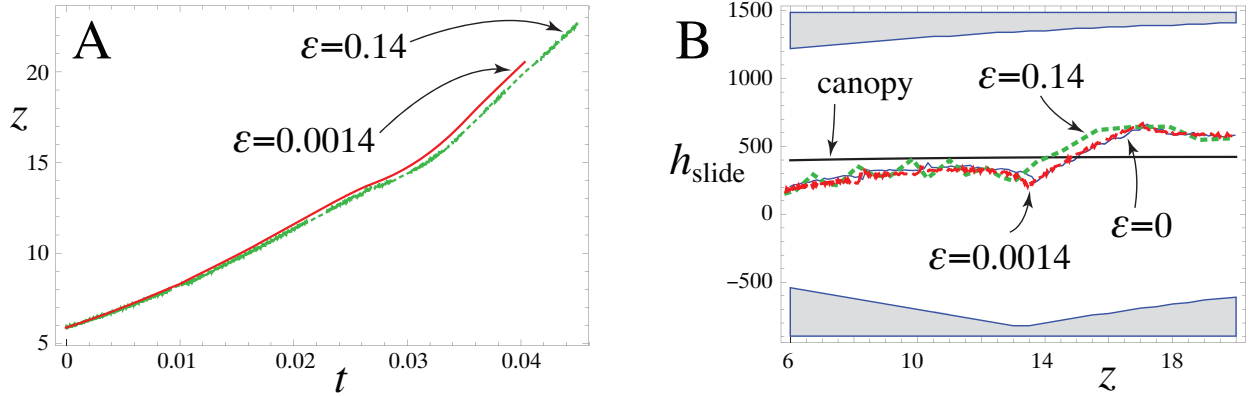}
      	\caption{Simulation of jitter in the circuit model for the parameters in Fig. \ref{Fig:circjitter}. Panel A shows a solution of the system (\ref{circuit}) with $\eps=\sqrt2/10^3$ (full curve) and $\eps=\sqrt2/10$ (dotted curve). Panel B shows the speed along the switching intersection, calculated from the gradient of the full simulation for $\eps=\sqrt2/10^3$ (dashed green curve) and $\eps=\sqrt2/10$ (dotted red curve), and the ideal $\eps=0$ (thin blue curve) calculated by iterating the hysteretic map at fixed $z$. The $\eps=\sqrt2/10^3$ and $\eps=0$ curves are almost indistinguishable. The canopy from (\ref{eq:canopy}) is also shown, and the unshaded region indicates the convex hull. 
}
      	\label{Fig:circsim}
	\end{center}
\end{figure}	

The results agree with the predicted theory, showing jitter in the sliding vector field, and here we see how this affects also the system's orbit. The deviation in the orbit itself is only slight here over the small timescale shown, for physically realistic parameter values, but shows a marked effect in the gradient. This explores a significant portion of the convex hull, agreeing with the theoretical sliding motion given by the attractor in the chatterbox, deviating significantly from the canopy. 

In the simulation with $\eps=\sqrt2/10$, the switching frequency for $u_1$ is found to be more variable than the switching frequency of $u_2$, but for both it remains in a range between $5 kHz$ and $25 kHz$, which are within the operating range of widely used electronic sensors and semiconductor switches.


\section{Closing remarks}
\label{sec:conc}
\setcounter{equation}{0}

The phenomenon of jitter along the intersection of multiple switches is a consequence of the rich dynamics that arises when switches are not ideal discontinuities, but instead involve elements of hysteresis, time-delay, or discretization. 
If we can assume that each switching multiplier or `duty ratio' $\lambda_i$ 
is determined independently of the others, then their values are given by the canopy combination. This appears to be a good approximation for a system where the switch is a limit of a smooth sigmoid function which becomes infinitely steep, or of a noisy switch as the noise amplitude tends to zero. The canopy appears to be a poor approximation when hysteresis, time-delay, or discretization dominate the dynamics of switching, when instead the system evolves onto an attractor that determines the dynamics (and the $\lambda_i$'s), an attractor whose identity is sensitive to parameters of the vector fields and the switching model. 

The attractor observed in the presence of hysteresis, time-delay, or discretization, is not obvious a priori from the vector fields, but is the solution of a continuous piecewise-differentiable circle map in the case of hysteresis (and more complicated maps in the cases of time-delay and discretization). The attractor undergoes bifurcations as parameters of the system are varied. These parameters may belong to the vector fields themselves, or to the regularisation. The bifurcations cause abrupt jumps in the attractor, resulting in abrupt jumps or `jitter' of the dynamics along an intersection of switches. The investigation in \cite{AlSe99}, which seems to be not yet widely known, showed that the sliding speed depended on the precise attractor at the intersection, the solution of a circle map, and we have added to the explanation and details of the phenomenon here, particularly its observable effect on the dynamics. The effect would appear to be highly significant 
both for theoretical and practical problems involving interactions between two or more switches. 

\section*{Bibliography}

\providecommand{\newblock}{}


\begin{thebibliography}{10}
\expandafter\ifx\csname url\endcsname\relax
  \def\url#1{{\tt #1}}\fi
\expandafter\ifx\csname urlprefix\endcsname\relax\def\urlprefix{URL }\fi
\providecommand{\eprint}[2][]{\url{#2}}
\setlength{\itemsep}{0pt}

\bibitem{OeHi96}
Oestreich M, Hinrichs N and Popp K 1996 {\em Arch. Appl. Mech.\/} {\bf 66}
  301--314

\bibitem{TaLa12}
Tan S~C, Lai Y~M and Tse C 2012 {\em Sliding Mode Control of Switching Power
  Converters.\/} (Boca Raton, FL: CRC Press)

\bibitem{Su06}
Sun J~Q 2006 {\em Stochastic Dynamics and Control.\/} ({\em Nonlinear Science
  and Complexity.\/} vol~4) (Amsterdam: Elsevier)

\bibitem{Fi88}
Filippov A 1988 {\em Differential Equations with Discontinuous Righthand
  Sides.\/} (Norwell: Kluwer Academic Publishers.)

\bibitem{Je14d}
Jeffrey M 2014 {\em Phys. D\/} {\bf 273-274} 34--45

\bibitem{UtGu99}
Utkin V, Guldner J and Shi J 1999 {\em Sliding Mode Control in
  Electro-Mechanical Systems.\/} (Boca Raton, FL: CRC Press)

\bibitem{WoSt08}
Wojewoda J, Stefa\'{n}ski A, Wiercigroch M and Kapitaniak T 2008 {\em Phil.
  Trans. R. Soc. A\/} {\bf 366} 747--765

\bibitem{OlAs98}
Olsson H, {\AA}str\"{o}m K, Canudas~de Wit C, G\"{a}fvert M and Lischinsky P
  1998 {\em Eur. J. Control\/} {\bf 4} 176--195

\bibitem{KaEn13}
Kaper H and Engler H 2013 {\em Mathematics and Climate.\/} (Philadelphia: SIAM)

\bibitem{KuRi03}
Kuznetsov Y, Rinaldi S and Gragnani A 2003 {\em Int. J. Bifurcation Chaos\/}
  {\bf 13} 2157--2188

\bibitem{PiPo14}
Piltz S, Porter M and Maini P 2014 {\em SIAM J. Appl. Dyn. Sys.\/} {\bf 13}
  658--682

\bibitem{JeCh10}
Jeffrey M, Champneys A, di~Bernardo M and Shaw S 2010 {\em Phys. Rev. E\/} {\bf
  81} 016213

\bibitem{SoTe96}
Sotomayor J and Teixeira M 1996 regularisation of discontinuous vector fields.
  {\em Proceedings of the International Conference on Differential Equations,
  Lisboa.\/} pp 207--223

\bibitem{TeDa12}
Teixeira M and da~Silva~PR 2012 {\em Phys. D\/} {\bf 241} 1948--1955

\bibitem{NoJe15}
Novaes D and Jeffrey M 2015 {\em J. Diff. Eq.\/} {\bf 259} 4615--4633

\bibitem{BuOu09}
Buckdahn R, Ouknine Y and Quincampoix M 2009 {\em Bull. Sci. Math.\/} {\bf 133}
  229--237

\bibitem{SiKu14b}
Simpson D and Kuske R 2014 {\em Stoch. Dyn.\/} {\bf 14} 1450010

\bibitem{GlKo10b}
Glendinning P and Kowalczyk P 2010 {\em Phys. D\/} {\bf 239} 58--71

\bibitem{Si14c}
Simpson D 2014 {\em Discrete Contin. Dyn. Syst.\/} {\bf 34} 3803--3830

\bibitem{Je16}
Jeffrey M 2016 {\em Int. J. Bifurcation Chaos\/} {\bf 26} 1650087

\bibitem{AlSe99}
Alexander J and Seidman T 1999 {\em Houston, J. Math.\/} {\bf 25} 185--211

\bibitem{AlSe98}
Alexander J and Seidman T 1998 {\em Houston, J. Math.\/} {\bf 24} 545--569

\bibitem{Je14c}
Jeffrey M 2014 {\em SIAM J. Appl. Dyn. Syst.\/} {\bf 13} 1082--1105

\bibitem{Di15}
Difonzo F 2015 {\em The {F}ilippov moments solution on the intersection of two
  and three manifolds.\/} Ph.D. thesis Georgia Institute of Technology

\bibitem{DiDi15}
Dieci L and Difonzo F 2015 The moments sliding vector field on the intersection
  of two manifolds. to appear: {\em J.~Dyn.~Diff.~Equat.}

\bibitem{lopez}
Dieci L and Lopez L 2011 
{\em Numer. Math.} {\bf 117} 779-811

\bibitem{Ar88}
Arnol'd V 1988 {\em Geometrical Methods in the Theory of Ordinary Differential
  Equations.\/} (New York: Springer-Verlag)

\bibitem{YaHa87}
Yang W~M and Hao B~L 1987 {\em Comm. Theoret. Phys.\/} {\bf 8} 1--15

\bibitem{SiMe09}
Simpson D and Meiss J 2009 {\em Nonlinearity\/} {\bf 22} 1123--1144

\bibitem{Si15c}
Simpson D 2015 The structure of mode-locking regions of piecewise-linear
  continuous maps. {\em Unpublished}

\bibitem{MaMa93}
Maistrenko Y, Maistrenko V and Chua L 1993 {\em Int. J. Bifurcation Chaos.\/}
  {\bf 3} 1557--1572
  
  \bibitem{mach}
  Gabovich, A 2012 {\em Superconductors - Materials, Properties and Applications} (InTech)

\bibitem{NuYo95}
Nusse H and Yorke J 1995 {\em Int. J. Bifurcation Chaos.\/} {\bf 5} 189--207

\bibitem{DiBu08}
di~Bernardo M, Budd C, Champneys A and Kowalczyk P 2008 {\em Piecewise-smooth
  Dynamical Systems. Theory and Applications.\/} (New York: Springer-Verlag)

\bibitem{GuHa15}
Guglielmi N and Hairer E 2015 {\em SIAM J. Appl. Dyn. Syst.\/} {\bf 14}
  1454--1477

\bibitem{Je15e}
Jeffrey M 2015 Exit from sliding in piecewise-smooth flows: deterministic
  vs.~determinacy-breaking. {\em Unpublished.}

\bibitem{EdGl14}
Edwards R and Glass L 2014 {\em Amer. Math. Monthly\/} {\bf 121} 793--809

\bibitem{AsIz89}
Asarin E and Izmailov R 1989 {\em Automat. Remote Contr.\/} {\bf 50} 1181--1185
  translation of {\em Avtomatika i Telemekhanika}, 9:43-48, 1989

\bibitem{DoLe92}
Dontchev A and Lempio F 1992 {\em SIAM Rev.\/} {\bf 34} 263--294

\bibitem{BaBa82}
Bafico R and Baldi P 1982 {\em Stochastics\/} {\bf 6} 279--292

\bibitem{FlLa08}
Flandoli F and Langa J 2008 {\em Stoch. Dyn.\/} {\bf 8} 59--75

\bibitem{BoSu10}
Borkar V and Suresh~Kumar K 2010 {\em J. Theor. Probab.\/} {\bf 23} 729--747

\bibitem{SiJe16}
Simpson D and Jeffrey M 2016 {\em Proc. R. Soc. A\/} {\bf 472} 20150782

\bibitem{StVa69}
Stroock D and Varadhan S 1969 {\em Comm. Pure Appl. Math.\/} {\bf 22} 345--400

\bibitem{KrRo05}
Krylov N and R\"{o}ckner M 2005 {\em Probab. Theory Relat. Fields\/} {\bf 131}
  154--196

\bibitem{Sk89}
Skorokhod A 1989 {\em Asymptotic Methods in the Theory of Stochastic
  Differential Equations.\/} (Providence: American Mathematical Society)

\bibitem{Kh10}
Khasminskii R 2010 {\em Stochastic Stability of Differential Equations.\/} (New
  York: Springer)

\bibitem{CaBa15}
Capasso V and Bakstein D 2015 {\em An Introduction to Continuous-Time
  Stochastic Processes.\/} (New York: Birkh\"{a}user)
  
  \bibitem{Utkin1993}
V.~I. Utkin, ``{Sliding mode control design principles and applications to
  electric drives},'' {\em IEEE Transactions on Industrial Electronics},
  vol.~40, no.~1, pp.~23--36, 1993.

\bibitem{Lai2008}
S.-C. Tan, Y.~M. Lai, and C.~K. Tse, ``{General Design Issues of Sliding-Mode
  Controllers in DC-DC Converters},'' {\em IEEE Transactions on Industrial
  Electronics}, vol.~55, pp.~1160--1174, mar 2008.

\bibitem{Sira-Ramirez1987}
H.~Sira-Ramirez, ``{Sliding motions in bilinear switched networks},'' {\em IEEE
  Transactions on Circuits and Systems}, vol.~34, pp.~919--933, aug 1987.

\end{thebibliography}

\end{document}